\documentclass[11pt]{amsart}
\usepackage[latin9]{inputenc}
\usepackage{color}
\usepackage{cancel}
\usepackage{amsbsy}
\usepackage{amstext}
\usepackage{amsthm}
\usepackage{amssymb}
\usepackage{mathtools}
\PassOptionsToPackage{normalem}{ulem}
\usepackage{ulem}

\makeatletter
\numberwithin{equation}{section}
\numberwithin{figure}{section}

\usepackage{tikz-cd}
\usepackage{fancyhdr}
\usepackage{latexsym}
\usepackage{amscd}\usepackage{amsfonts}\usepackage{pstricks}
\usepackage{young}
\usepackage{enumitem}
\usepackage{amsthm}
\usepackage{mathrsfs}

\usepackage{ifthen}
\usepackage{longtable}
\usepackage{todonotes}
\usepackage{marginnote}

\renewcommand{\subsection}[1]{\vspace{3mm}\refstepcounter{subsection}\noindent{\bf \thesubsection. #1.} }

\renewcommand{\subsubsection}[1]{\vspace{3mm}\refstepcounter{subsubsection}\noindent{\bf \thesubsubsection. #1.} }

\numberwithin{equation}{section}

\newtheorem{theorem}{Theorem}
\newtheorem{lemma}[theorem]{Lemma}

\newtheorem{proposition}[theorem]{Proposition}

\theoremstyle{definition}

\newtheorem{definition}[theorem]{Definition}

\theoremstyle{remark*}
\newtheorem*{remark*}{Remark}

\newtheorem*{example*}{Example}

\newtheorem*{thm}{Main Theorem}

\def\PP{\mathbb P}

\DeclareMathOperator{\GL}{GL}

\def\min{\mathop{\mathrm{min}}}

\def\PP{\mathbb P}

\def\cal{\mathcal }

\def\p{\mathbf p}

\newcommand{\bfu}{\mathbf{u}}

\newcommand{\exc}{\operatorname{exc}}

\makeatother

\begin{document}
\title[Vojta's abc   conjecture for entire curves in  toric varieties ]{Vojta's abc conjecture for entire curves  in  toric varieties highly ramified over the boundary}
 
\author{Min Ru}
\address{
 Department  of Mathematics\newline
\indent University of Houston\newline
\indent Houston,  TX 77204, U.S.A.} 
\email{mru2@central.uh.edu}
\author{Julie Tzu-Yueh Wang}
\address{Institute of Mathematics, Academia Sinica \newline
\indent No.\ 1, Sec.\ 4, Roosevelt Road\newline
\indent Taipei 10617, Taiwan}
\email{jwang@math.sinica.edu.tw}
\thanks{2020\ {\it Mathematics Subject Classification}: Primary 32H30; Secondary 32Q45 and 30D35}
\thanks{The research of Min Ru is supported in part by Simon Foundations grant award \#531604 and \#521604.
The research of  Julie Tzu-Yueh Wang was supported in part by Taiwan's NSTC grant  113-2115-M-001-011-MY3.}

\begin{abstract} We prove Vojta's  abc conjecture  for  projective space ${\Bbb P}^n({\Bbb C})$,  assuming  that the entire curves in ${\Bbb P}^n({\Bbb C})$ are highly ramified over the coordinate
hyperplanes. This extends the results of Guo Ji and the second-named author for the case  $n=2$ (see \cite{GW22}).    We also explore the corresponding results  for   projective toric varieties.  Consequently, we establish a version of Campana's orbifold conjecture for  finite coverings of projective toric varieties. 
 \end{abstract}

\maketitle
\baselineskip=16truept

 \section{Introduction}\label{sec:intro}
  Let $X$ be an algebraic variety and  $D$ be an effective divisor (possibly empty) on $X$.   The set $X\backslash D$ is {\it Brody hyperbolic}  
if every entire curve $f:\mathbb C\to X\backslash D$  is constant.   It is  {\it Brody  quasi-hyperbolic}  if 
every entire curve $f:\mathbb C\to X\backslash D$  is degenerate, meaning  its image is contained a proper Zariski-closed subset $Z$ of $X$.   If such a subset $Z$ can be chosen independently of $f$, then
$X\backslash D$ is said to be {\it Brody strong quasi-hyperbolic} and  $Z$   is called an  {\it exceptional set}.    This leads us to the Griffiths-Lang's conjecture.

\noindent{\bf Griffiths-Lang's conjecture}. {\it   Let $X$ be a smooth complex projective variety,  $D$  a normal crossing divisor on $X$, and $K_X$  a  canonical divisor on $X$.  If $K_X+D$ is big, then $X\backslash D$ is Brody  quasi-hyperbolic.}

A more precise quantitative version,  in terms of the Nevanlinna's functions, can be formulated as follows.

 \noindent{\bf Griffiths' conjecture}. {\it  For a holomorphic map  $f:\mathbb C\to X$ with Zariski dense image,  the inequality 
 $$ T_{K_X+D,f}(r) \leq_{\rm exc} N_f(D, r) +\epsilon T_{ A,f}(r)$$ 
 holds for any $\epsilon >0$, where  $A$ is an ample divisor on $X$, and  the notation}    $\leq_{\rm exc}$ means the inequality holds  for all $r\in (0, \infty)$ except for a set  of finite measure. 

 Here, 
for each positive integer $n$, $N_f^{(n)}(D,r)$ is the $n$-truncated counting function with respect to $D$   given by
	$$
		N_f^{(n)}(D,r)=\sum_{0<|z|<r}\min\{{\rm ord}_z f^*D, n\}\log\frac{r}{|z|}+\min\{{\rm ord}_0 f^*D, n\}\log r, $$
		$N_f(D,r)$ is the  counting function with respect to $D$   given by $N_f^{(\infty)}(D,r)$,
	and
	$T_{D,f}(r)$   is the (Nevanlinna) height function relative to the divisor $ D$ (referring to \cite[Section 12]{Vojta}).

Paul Vojta formulated the following general abc conjecture (see  \cite[Conjecture 15.2]{Vojta} and \cite[Conjecture 23.4]{Vojta}), which extends the above Griffiths' conjecture. It is also 
 viewed as the   quantitive version of Campana's orbifold conjecture which will be stated later in the introduction. 

 \noindent{\bf Vojta's  abc conjecture}. {\it      Let $X$ be a smooth complex projective variety,  $D$  a normal crossing divisor on $X$, and $K_X$  a canonical divisor on $X$. 
 For any $\epsilon>0$, there exists a proper Zariski-closed subset $Z$ of $X$ such that  for any entire curve $f:\mathbb C\to X$ whose image is not contained in $Z$, the inequality
 $$ T_{K_X+D,f}(r) \le_{\rm exc}
 N_f^{(1)}(D,r)+ \epsilon T_{ A,f}(r)$$ holds for any $\epsilon>0$, 
 where $A$ is an ample divisor on $X$. 
}

  Note that Vojta's  abc conjecture strengthens Griffiths' conjecture in two key ways: (i) the exceptional set $Z$ above is independent of the map 
$f$, and (ii) the counting function is truncated at level one. Proving these two features is quite challenging, even if the aforementioned version of Griffiths' conjecture could be established. In this paper, we first prove Vojta's abc conjecture for entire curves in ${\Bbb P}^n({\Bbb C})$, 
assuming the map is highly ramified over the coordinate hyperplanes. The precise statement of our main result is  stated below. Additionally, we highlight that the exceptional set obtained in the theorem can be determined explicitly-this is  a challenging issue in the study of Nevanlinna theory and complex hyperbolicity.

 \begin{thm}\label{thm}
Let $G$ be a non-constant  homogeneous polynomial   in $\mathbb C[x_0,\dots, x_n]$ with no monomial factors and no repeated factors.   
Assume that  $[G=0]$  is in general position with the coordinate hyperplanes $H_i=[x_i=0]$ in $\mathbb P^n(\mathbb C)$, $0\le i\le n$.  
 Then, for any $\epsilon >0$, there exist  an effectively computable positive integer $\ell$ and a proper Zariski closed subset $Z$ of $\mathbb{P}^n({\Bbb C})$  that can be  explicitly determined, such that the following two inequalities hold: 
\begin{itemize}
	\item[\rm (i) ] $N_{G(\mathbf{f})}(0,r) -N^{(1)}_{G(\mathbf{f})}(0,r)\le_{\exc} \epsilon T_{f}(r)$, and 
	\item[\rm (ii) ] $ (\deg  G- \epsilon)T_{f}(r) \le_{\exc} N^{(1)}_{G(\mathbf{f})}(0,r)$,
\end{itemize} 
  for any non-constant  holomorphic map  $f: {\Bbb C} \rightarrow {\Bbb P}^n({\Bbb C})$
   whose  image  is not contained in $Z$, and where $ {\rm mult}_{z_0}(f^*H_i) \ge   \ell$ for all $i$ with $0\leq i\leq n$ and 
for every $z_0 \in {\Bbb C}$  with $f(z_0) \in  H_i$. 
Here, the evaluation of $G$ is taken at $\mathbf{f}=(f_0,\hdots,f_n)$ in a reduced form of $f$,  meaning that  $f_0,\hdots,f_n$ are entire functions without common zeros. 
  \end{thm}
  We note that this  theorem corresponds to a case of  Vojta's abc conjecture for entire curves with $D=H_0+\cdots+H_{n}+[G=0]$, so that  $K_{\mathbb P^n+D}$ is linearly equivalent to $[G=0]$. 
The case  $n=2$ was previously established by   Ji Guo and the second named author (\cite[Theorem 1.4]{GW22}). 
 The new idea in this paper is motivated by  
 the recent paper  \cite{CHSX}, where the authors  considered  the parabolic Riemann surface $\mathcal Y_f:=\mathbb C\setminus{f^{-1}(D)}$,  reducing the problem  to the omitting case.  To achieve this in our setting, we need to extend the results in \cite{GNSW}  and \cite{GW22}  to the parabolic setting. 
 
  We state some consequences of our Main Theorem,  starting with the  introduction of Campana's orbifolds.  Instead of considering the complement case in Brody's hyperbolicity, we can also examine an entire curve  $f:\mathbb C\to X$ that is ramified over $D$. According to Campana \cite{Camp}, 
	 an {\it orbifold divisor} is  defined as  $\Delta: =\sum_{Y\subset X} (1-m_{\Delta}^{-1}(Y))\cdot Y$, 
	 where $Y$ ranges over all irreducible divisors of $X$,   and $m_{\Delta}(Y ) \in [1, \infty)  \cap {\Bbb Q}$, with all but finitely many values equal to one.
	 An {\it orbifold entire curve} $f: {\Bbb C}\rightarrow  (X,\Delta)$, with $\Delta=\sum_i (1-m_i^{-1}) D_i$,  is a homomorphic  map such that $f({\Bbb C})\not\subset \mbox{Supp}(\Delta)$ and ${\rm mult}_{z_0}(f^*D_i) \ge  m_i$ for all $i$ and 
for every $z_0 \in {\Bbb C}$  with $f(z_0) \in  D_i$. 

 \noindent{{\bf Campana's orbifold conjecture}.  {\it If $(X,\Delta$) is an orbifold pair such that $K_X + \Delta$ is big, then there exists a proper closed subvariety $Z \subset X$ containing all the non-constant orbifold entire curves $f: {\Bbb C} \rightarrow  (X, \Delta)$. }

\noindent{\bf Remarks}.
\begin{enumerate}
\item The Campana's conjecture is a consequence of the above Vojta's  abc conjecutre.
\item  Given an orbifold divisor $\Delta$ on $X$,  let  $D:={\rm Supp}(\Delta)= D_1+\cdots+D_{q}$.
Recall that the pair $(X, D)$ is said to be   of {\it log-general type} if the divisor $K_X+D$ is big.
		It is clear that  $(X, D)$ is of log-general type if  $(X,\Delta)$ is   of  general type, i.e. $K_X + \Delta$ is big.
		On the other hand, the condition that $(X, D)$ is of log-general type implies that $(X,\Delta)$ is   of  general type if $ (X,\Delta) $ has a sufficiently large multiplicity along  each $D_i$, $1\leq i\leq q$.  $($See \cite[Corollary 2.2.24]{LazarsfeldI}.$)$
\end{enumerate}

 The Main Theorem leads to the following two theorems concerning Campana's orbifold conjecture; however, we will not include the proofs, as they can be adapted from those in \cite{GW22} for the case $n=2$.  We begin with the result concerning  $n+2$ components.
\begin{theorem} Let $\Delta_0$  be an orbifold divisor on ${\Bbb P}^n({\Bbb C})$ and let  $H_0 , H_1 , \dots,  H_n$ be the coordinate hyperplanes in ${\Bbb P}^n({\Bbb C})$, such that  $\Delta_0$ and $H_0, H_1, \dots, H_n$ are in general position. Let $m_i \in  (1, \infty)\cap {\Bbb Q}$,
$0\leq i\leq n$, and $\Delta=\Delta_0+(1-{1\over m_0} )H_0+(1-{1\over m_1} )H_1+\cdots + (1-{1\over m_n} )H_n$. Assume that $\deg\Delta>n+1$. 
Then there exist a proper Zariski closed subset $W$ of ${\Bbb P}^n({\Bbb C})$ and an effectively computable positive integer $\ell$ such that the image of any orbifold entire curve $f: {\Bbb C} \rightarrow ({\Bbb P}^n({\Bbb C}), \Delta)$ with $\min\{m_0, m_1,  \dots, m_n\} \ge \ell$  must be contained in $W$. 
\end{theorem}
The next  case involves $n+1$ components, not all    are  hyperplanes, in $\mathbb{P}^n({\Bbb C})$.\begin{theorem}\label{GG_conj}
		Let 
		$F_i$, $1\le i\le n+1$, be homogeneous irreducible polynomials of positive degrees  in $\mathbb{C}[x_0,\hdots,x_n]$.
		Assume that $D_i:=[F_i=0]\subset\mathbb P^n(\mathbb C)$,  $1\le i\le n+1$,   intersect transversally.  Let $\Delta=(1-{1\over m_1} )D_1+\cdots + (1-{1\over m_{n+1}} )D_{n+1}$ be an orbifold divisor of $\mathbb{P}^n({\Bbb C})$.  Assume that  $\deg \Delta>n+1$.
		Then, there exist  an effectively computable  positive integer  $\ell$ and a proper Zariski closed subset $W$ of $\mathbb P^n({\Bbb C})$ such that the image of any orbifold entire curve $f: {\Bbb C} \rightarrow ({\Bbb P}^n({\Bbb C}), \Delta)$ with $\min\{m_1,  \dots, m_{n+1}\} \ge \ell$  must be contained in $W$. 
	\end{theorem}

Note  that $\mathbb P^n({\Bbb C})$ is a  toric variety  when we identify  ${\mathbb G}_m^n$ with ${\Bbb P}^n({\Bbb C})\backslash \cup_{i=0}^n H_i$, where $H_0, \dots, H_n$ are the $n+1$ coordinate hyperplanes.  A  toric variety $X$ is defined as  a variety over ${\Bbb C}$ that contains ${\mathbb G}_m^n$
as a dense open subvariety, with  the action of ${\mathbb G}_m^n$ on itself extends to an algebraic
action of ${\mathbb G}_m^n$ on $X$. We refer to \cite{cox} for more details on toric varieties.  Using the arguments in \cite{levin_gcd} and \cite{GGW}, we can extend our Main Theorem to projective toric varieties.
Recall from \cite{levin_gcd} that an  ``admissible pair"  is a couple $(X,V)$ where  $V$ is a nonsingular variety embedded in a nonsingular projective variety $X$ in such a way that $D_0 =X\setminus V$ is a normal crossings divisor. A Campana orbifold $(X,\Delta)$ is said to be {\it associated to the admissible pair $(X,V)$} if  $D_0 $ is the support of $\Delta$.    

 \begin{theorem}\label{toric}
Let $X$ be a nonsingular projective toric  variety and 
 let $D$ be an effective reduced divisor on the admissible pair $(X,\mathbb G_m^n)$, such that $D$ is in general position with the boundary $D_0 =X\setminus \mathbb G_m^n$. Let $A$ be a big divisor on $X$.
 Then, for every  $\epsilon >0$, there is a positive  integer  $\ell $ and a proper Zariski-closed  subset $Z$ of $X$ such that, for any Campana orbifold $(X,\Delta)$ associated to the admissible pair  $(X,\mathbb G_m^n)$ with multiplicities at least $\ell$ along the boundary $D_0$,  the following two inequalities hold:

\begin{enumerate}
\item
$N_f(D,r)-N_f^{(1)}(D,r)\le_{\rm exc} \epsilon T_{A,f}(r)$, and 
\item $T_{D,f}(r) \le_{\rm exc} N_f^{(1)}(D,r) +\epsilon T_{A,f}(r) $ 
\end{enumerate}
for any  non-constant entire orbifold curve $f$ on $(X,\Delta)$ such that the image of $f$ is not contained in $Z$.
\end{theorem}

Theorem \ref{toric} implies Campana's conjecture for an orbifold pair  that is a finite covering of a  toric variety, as follows.
\begin{theorem}\label{finitemorphismtoric}
Let $X$ be a nonsingular projective toric  variety, and let  $D_0 =X\setminus \mathbb G_m^n$.
Let  $Y$ be a nonsingular  complex   projective  variety with a finite  morphism $ \pi:  Y\to X$.  
Let $H:=\pi^*(D_0)$ and $R\subset Y$  the ramification divisor of $\pi$ omitting components from the support of $H$.
Assume that  $\pi(R)$ and $D_0$ are in general position on $X$. 		
Then there exists a positive integer $\ell$ such that, if the orbifold pair $(Y,\Delta)$ is   of  general type with ${\rm Supp}(\Delta)=  {\rm Supp}(H)$ and  multiplicity at least $\ell$ for each component, then there exists a proper Zariski closed subset $W$ of $Y$ such that  the image of any non-constant orbifold entire curve $f: {\Bbb C} \rightarrow (Y,\Delta)$  must be contained in $W$. 
	\end{theorem}
	
The paper is organized as follows. In Section \ref{motivation}, we  prove a version of the Main Theorem for rational maps and explain the motivation for using  parabolic Riemann surfaces.  	
Section \ref{parabolic} reviews basic results from   Nevanlinna theory relevant to  parabolic settings,  as well as  the key theorems  from \cite{CHSX}. Furthermore, we  extend the GCD theorem developed in \cite{levin2019greatest} and  in \cite{GW22} to parabolic setting. In Section \ref{abcparabolic}, we establish the abc theorem for parabolic Riemann surfaces, building on the  methods developed in \cite{GNSW}. Section \ref{proofmain} presents the proof of the Main Theorem, while Section  \ref{toricSection} extends results to toric varieties.

\section{Motivation through the rational case}\label{motivation} 
In \cite{GNSW}, a version of Vojta's  abc conjecture for algebraic tori over algebraic function fields is established. As we will see, the algebraic reduction method presented in   \cite{GNSW}   is a  key technical component in proving
 our Main Theorem. The core idea behind the proof is also inspired by the function field approach. In this section, we will prove a version of the Main Theorem for rational functions and explain the role of parabolic Riemann surfaces in this context.  

\subsection{The Rational Case}\label{rationalcase}
 The statements and arguments apply verbatim to algebraic function fields (see also \cite[Theorem 1.13]{GGW}).
Here, let $K=\mathbb C(t)$ be the function field of $\mathbb P^1({\Bbb C})$.  Let $S$ be a finite subset of  $\mathbb P^1(\mathbb C)$.  We denote by  ${\mathcal O}_{S}^*$  the set of $S$-units, which consists of rational functions with no zeros and poles outside of $S$.  For a rational function $f=\frac{P}{Q}$, where $P$ and $Q$ are coprime polynomials, we define $\deg  f:=\max\{\deg P,\deg Q\}$ and denote by ${\rm ord}_{\p}^+(f)$ the order of the zero of $f$ at $\p\in\mathbb P^1(\mathbb C)$.
We use the notation  $N_{S}(f)$ (respectively $N_{S}^{(1)}(f)$)  to count the number of zeros of $f$ with multiplicity (respectively, without multiplicity) outside of $S$.  When $S$ is empty, we simply write  $N(f) $,  which equals $\deg f$, to count the number of zeros with multiplicity  on $\mathbb P^1(\mathbb C)$; $N^{(1)}(f)$ denotes the same count without multiplicity.

  Recall a simplified version of the main technical theorem in  \cite{GNSW}. 
 \begin{theorem}[{\cite[Theorem 4]{GNSW}}]\label{GNSW}
Let  $G\in \mathbb C[ x_1,\hdots, x_n]$ be a non-constant    homogeneous polynomial  with neither monomial  factors nor repeated factors. 
Let $S$ be a finite subset of  $\mathbb P^1(\mathbb C)$.  Then, for any $\epsilon>0$,   there exist a positive real number $c_0$  and  
a proper Zariski closed subset $Z$ of $\mathbb{A}^n$ such that, for all  $\mathbf u:=(u_1,\hdots,u_n)\in ({\cal O}_{S}^*)^{n}\setminus Z$, we have either
\begin{enumerate}
\item the inequality 
\begin{align*}
 \max_{1\le i\le n}\{ \deg (u_i)\}\le  c_0   \max\{1,|S|-2\}  
\end{align*}
holds, 
  \item or the following two statements hold

 \begin{enumerate}
 \item[{\rm(a)}]  $N_{S}( G(\mathbf u) )-N_{S}^{(1)}( G(\mathbf u))\le \epsilon \max_{1\le i\le n}\{\deg  u_i\}$ if $G(u_1,\hdots,u_n)\ne 0$.
 \item[{\rm(b)}]  If $G(0,\hdots,0)\ne 0$ and $\deg_{X_i}G=\deg G=d$ for $1\le i\le n$, then 
 \begin{align*} 
 N_{S}^{(1)}( G(\mathbf u))\ge  (\deg G -\epsilon) \max_{1\le i\le n}\{\deg  u_i\}.
 \end{align*}
  \end{enumerate}
  \end{enumerate}
 Here $c_0$ can be  effectively bounded from above in terms of $\epsilon$, $n$, and the degree of $G$.    Moreover, the exceptional set $Z$ can be expressed as the zero locus of a finite set
$\Sigma\subset \mathbb C[x_1,\ldots,x_n]$ with the following properties: {\rm(Z1)} $\Sigma$ depends on $\epsilon$ and $G$  but independent of $S$ and can be determined explicitly and  {\rm(Z2)} $\vert \Sigma\vert$ and the degree of each polynomial in $\Sigma$ can be effectively bounded from above in terms of $\epsilon$, $n$, and the degree of $G$.
 \end{theorem}
 The construction of the set of polynomials $\Sigma$ follows from the specialization lemmas \cite[Lemma 15 and 16]{GNSW}.  In particular, it is evident that $\Sigma$ is independent of the choice of  $S$.  Note that $\mathbf u:=(u_1,\hdots,u_n)\in ({\cal O}_{S}^*)^{n}\setminus Z$ means that $F(\mathbf u)\ne 0$ (a nontrivial function and hence not identically zero) for every $F\in \Sigma$.  This formulation is closely analogous to the number field case.
Alternatively, one may view $\mathbf u$ as a morphism: $\mathbb P^1(\mathbb C)\setminus S \to \mathbb C^n$; then the condition $F(\mathbf u)\ne 0$ for every $F\in \Sigma$ is equivalent to saying that the image $\mathbf u(\mathbb P^1(\mathbb C)\setminus S)$  is not contained in $Z$.

We will demonstrate how to apply the above theorem to show the following version of Main Theorem for rational case.

 \begin{theorem}\label{rational}
Let $G$ be a non-constant  homogeneous polynomial   in $\mathbb C[x_0,\cdots, x_n]$ with no monomial factors and no repeated factors.    
Assume  $[G=0]$  is in general position with the coordinate hyperplanes $H_i=[x_i=0]$, $0\le i\le n$. 
Then, for any $\epsilon >0$, there exist  effectively computable positive integers $\ell$ and a proper Zariski closed subset $Z$ of $\mathbb{P}^n({\Bbb C})$ such that the following two inequalities hold: 
\begin{itemize}
	\item[\rm (i) ] $N (G(\mathbf{f}))  -N^{(1)} (G(\mathbf{f}))\le  \epsilon \deg \mathbf{f}$, and 
	\item[\rm (ii) ] $N^{(1)} (G(\mathbf{f}))\ge  (\deg  G- \epsilon)\deg \mathbf{f}$,
\end{itemize} 
for any non-constant entire rational curve  $\mathbf{f}=(f_0,\hdots,f_n)$, where $f_0,\hdots,f_n\in\mathbb C[z]$ are  polynomials without common zeros,  provided that  $\mathbf{f}(\mathbb C)\not\subset Z$ and that the following condition is satisfied:
for each $0\le i\le n$  and $\p\in \mathbb P^1({\Bbb C})$,  if ${\rm ord}_{\p} (f_i)>0$, then 
${\rm ord}_{\p} (f_i)\ge \ell$.  
 Here,  $\deg \mathbf{f}=\max\{\deg f_0,\hdots,\deg f_n\}$.
  \end{theorem} 
In addition to  Theorem \ref{GNSW}, we also require the following result to prove Theorem \ref{rational}.
\begin{proposition}[{\cite[Corollary 6]{GNSW}}] \label{ProximityAffine} 
Let $G=\sum_{\mathbf i\in I_G}a_{\mathbf i}{\mathbf x}^{\mathbf i}\in \mathbb C[ x_1,\hdots, x_n]$ be a non-constant polynomial, where  $a_{\mathbf i}\ne 0$ if $\mathbf i\in I_G$. Assume that $G(0,\dots,0)\ne 0$.  Let $Z$ be the Zariski closed subset that is the union of hypersurfaces of $\mathbb A^n$ of the form  $\sum_{\mathbf i\in I} a_{\mathbf i}{\mathbf x}^{\mathbf i}=0$ where $I$ is a non-empty subset of $I_G$.
 Let $S$ be a finite subset of  $\mathbb P^1(\mathbb C)$.  Then,
for all $(u_1,\hdots,u_n)\in (\mathcal O_S^*)^n\setminus Z$, we have 
 \begin{align*} 
 \sum_{\p\in S} {\rm ord}_{\p}^+ (G(u_1,\hdots,u_n)) \le  
\tilde c \max\{1, |S|-2\}  
  \end{align*}
  where $\tilde c=\frac12 \binom{n+\deg G}{n}  (\binom{n+\deg G}{n}+1)$.
\end{proposition}

 The proof of Theorem \ref{rational} (i)  can be outline as the follows.
Given a rational map  $\mathbf{f}=(f_0,\hdots,f_n)$, where $f_0,\hdots,f_n\in\mathbb C[z]$ are  polynomials without common zeros, we define a set $S_\mathbf{f}$ consisting of the zeros of all $f_i$.  With this choice, each coordinate function $f_i$ becomes an $S_\mathbf{f}$-unit, so that Theorem \ref{GNSW} (ii) can be applied to estimate the multiplicities of the zeros of $ G(\mathbf f)$ outside $S_\mathbf{f}$.  We then use Proposition \ref{ProximityAffine} to estimate the zeros of $ G(\mathbf f)$ inside $S_{\mathbf{f}}$.
Crucially, the constants and exceptional sets appearing in Theorem \ref{GNSW} and Proposition \ref{ProximityAffine} do not depend on $S$, allowing this method to be used for all rational maps. 
 
 \begin{proof}[Proof of Theorem \ref{rational}]
 Denote by $\ell$ a sufficiently large integer to be determined later.
  Let  $\mathbf{f}=(f_0,\hdots,f_n)$, where $f_0,\hdots,f_n\in\mathbb C[z]$ are  polynomials without common zeros. 
 By the assumption,   for each $0\le i\le n$  and  $\p\in \mathbb P^1({\Bbb C})$, if ${\rm ord}_{\p} (f_i)>0$, then 
${\rm ord}_{\p} (f_i)\ge \ell$. 
   Let
\begin{align}\label{setS}
\tilde S_\mathbf{f}:=\{\p\in \mathbb P^1(\mathbb C) \,|\, {\rm ord}_{\p}( f_i)>0\quad\text{for some $0\le i\le n$}\} \cup\{\infty\}.
\end{align}
Then
\begin{align}\label{sizeS}
|\tilde S_\mathbf{f}  |\le \sum_{i=0}^n  N_{S}^{(1)}(f_i)+ 1\le  \frac{n+1}{\ell} \max_{0\le i\le n}\{\deg f_i\} +1,
\end{align}
and $f_i\in\mathcal O_{\tilde S_\mathbf{f} }^*$ for   $0\le i \le n$.  
  Let  $\widetilde G(x_1,\hdots,x_n):=G(1, x_1,\hdots,x_n)$ and $u_i=\frac{f_i}{f_0}\in\mathcal O_{\tilde S_\mathbf{f}}^*$, $1\le i\le n$.  Then  $\widetilde G(u_1,\hdots,u_n):=f_0^{-\deg G}G({\mathbf f}).$ 
An elementary argument shows that 
\begin{align}\label{deg}
 \max_{1\le i\le n}\{ \deg u_i\}\le   \max_{0\le i\le n}\{ \deg f_i\}\le \sum_{j=1}^n \deg u_j.
\end{align}
 
Let $\epsilon>0$ be given. We now apply Theorem \ref{GNSW}   for  $\widetilde G\in\mathbb C[x_1,\hdots,x_n]$ and $\tilde S_{\mathbf{f}}$
  to find a  positive integer $c_0$  and  
a proper Zariski closed subset $Z_0$ of $\mathbb{A}^n $ independent of $\tilde S_\mathbf{f}$  such that   we have either the inequality
\begin{align}\label{htgbd}  
 \max_{0\le i\le n}\{ \deg f_i\}(=\deg \mathbf{f}) \le  c_0   \max\{1, |\tilde S_\mathbf{f}|-2\} 
 \end{align} holds 
or  the following two inequalities hold:
\begin{align}\label{mutiple1}  
  N_{\tilde S_\mathbf{f}}(G({\mathbf f}))-N_{\tilde S_\mathbf{f}}^{(1)}(G({\mathbf f}))=N_{\tilde S_\mathbf{f}}( \widetilde G(u_1,\hdots,u_n) )-N_{\tilde S_\mathbf{f}}^{(1)}( \widetilde G(u_1,\hdots,u_n))\le  \epsilon \deg \mathbf{f},
 \end{align} 
\begin{align}\label{truncate1}
  N_{\tilde S_\mathbf{f}}^{(1)}(G({\mathbf f}))=N_{\tilde S_\mathbf{f}}^{(1)}(\widetilde G(u_1,\hdots,u_n))\ge  (\deg  G- \epsilon)\cdot  \deg \mathbf{f},
 \end{align} 
 if $\mathbf{f} (\mathbb C)\not\subset \tilde Z_0$, where $\tilde Z_0$ is the projective closure of $Z_0$ in $\mathbb P^n({\Bbb C})$.
 Since the constant $c_0$ and the exceptional set $Z$ in Theorem \ref{GNSW} do  not depend on $S$, it is important to note that $c_0$ and $ \tilde Z_0$ do not depend on $\mathbf f$ as well.
We first show that \eqref{htgbd}  does not hold if $\ell>2c_0n$.  Indeed, if  \eqref{htgbd} holds  with $\ell>2c_0n$, then it follows from \eqref{sizeS} and the choice of $\ell$ that  \eqref{htgbd} yields
\begin{align*} 
 \max_{0\le i\le n}\{ \deg f_i\} 
\le  \frac{ (n+1)c_0}{\ell}  \max_{0\le i\le n}\{ \deg f_i\}  +c_0< \frac12\max_{0\le i\le n}\{ \deg f_i\}  +c_0.
\end{align*}
Therefore, $\max_{0\le i\le n}\{ \deg f_i\} <  2c_0.$  This is not possible since
  the $f_i$ are not all constant, and by our assumption that every zero has multiplicity at least $\ell$, we have $2nc_0<\ell <\max_{0\le i\le n}\{ \deg f_i\}$.

We continue to let  $\ell>2c_0n$.  As we have shown that \eqref{htgbd} does not hold, it follows that \eqref{mutiple1} and \eqref{truncate1} both hold if $\mathbf{f} (\mathbb C)\not\subset \tilde Z_0$.
Then it follows trivially from \eqref{truncate1} that 
 \begin{align} 
 N^{(1)} ( G(\mathbf{f}))\ge N_{\tilde S_\mathbf{f}}^{(1)}(G({\mathbf f}))\ge  (\deg  G- \epsilon)\cdot  \deg \mathbf{f}.
 \end{align}
This shows  (ii) holds. 

It remains to show  (i).  We now compute zeros for points in $\tilde S_\mathbf{f}$.  We will use Propostion \ref{ProximityAffine}.
Let $W$ be the Zariski closed subset   of $\mathbb P^n(\mathbb C)$ containing all hypersurfaces defined by all possible  subsums (including $G$) in the expansion of $G$.   Assume that  $\mathbf f(\mathbb C)\not\subset W$.
Since $\widetilde G(0,\dots,0)=G(1,0,\hdots,0)\ne 0$, we now apply Propositon  \ref{ProximityAffine} to get
 \begin{align*} 
 \sum_{\p\in \tilde S_\mathbf{f}} {\rm ord}^+_{\p}(\widetilde G(u_1,\hdots,u_n)) \le\tilde c   \max\{1, |\tilde S_\mathbf{f}|-2\},
  \end{align*}
where $\tilde c=\frac12 \binom{n+deg G}{n}  (\binom{n+\deg G}{n}+1)$.
Together with \eqref{sizeS} and that  $\ell\le \max_{0\le i\le n}\{ \deg f_i\}$, it yields
 \begin{align*} 
 m_{\tilde S_{\mathbf{f}}}(\widetilde G(u_1,\hdots,u_n) ):= \sum_{\p\in \tilde S_\mathbf{f}} {\rm ord}^+_{\p}(\widetilde G(u_1,\hdots,u_n))  \le  \frac{(n+1)\tilde c}{\ell}  \max_{0\le i\le n}\{ \deg f_i\} \le \epsilon \deg \mathbf{f} 
  \end{align*}
provided $\ell >(n+1)\tilde c_0\epsilon^{-1}.$
Combining with \eqref{mutiple1}, we have   
\begin{align}\label{mutiple0}  
 N (\widetilde G(u_1,\hdots,u_n) )-N^{(1)} ( \widetilde G(u_1,\hdots,u_n))\le  2\epsilon \deg \mathbf{f}.
 \end{align}
  We now repeat the above arguments for $(\frac{f_0}{f_i},\cdots,\frac{f_n}{f_i})$, $0\le i\le n$.  Then we find a Zariski closed subset $Z$ and a positive constant $c$, independent of $\mathbf f$, such that for each $0\le i\le n$ we have
\begin{align}\label{mutiplei}  
 N \left(G(\frac{f_0}{f_i},\cdots,\frac{f_n}{f_i}) \right)-N^{(1)} \left( G(\frac{f_0}{f_i},\cdots,\frac{f_n}{f_i})\right)\le  2\epsilon \deg \mathbf{f}. 
 \end{align}
Since $f_0,\hdots,f_n$ are  polynomials without common zeros, for each $z\in\mathbb C$ there exists at least one $f_i$ such that $f_i(z)\ne 0$.  Then 
$$
{\rm ord}^+_z(G(f_0,\hdots,f_n))={\rm ord}_z\left(G(\frac{f_0}{f_i},\cdots,\frac{f_n}{f_i})\right).
$$
This shows that 
 \begin{align*} 
 N(G(\mathbf{f}))  -N^{(1)} (G(\mathbf{f}))&\le \sum_{i=0}^n \left( N \left( G(\frac{f_0}{f_i},\cdots,\frac{f_n}{f_i}) \right)-N^{(1)} \left( G(\frac{f_0}{f_i},\cdots,\frac{f_n}{f_i})\right)\right)\\
 &\le  2(n+1)\epsilon \deg \mathbf{f}.   \quad\text{(by \eqref{mutiplei})}
 \end{align*} 
 The assertion   (i) can be derived by adjusting $\epsilon$.
\end{proof}

\subsection{Motivation for the proof of the Main Theorem}
 In Section~\ref{rationalcase}, the key to proving Theorem~\ref{rational} is the use of  the set $S$ (depending on the map) that contains all zeros and poles of each coordinate function of the map.  
The crucial observation is that all coordinate functions then become $S$-units, allowing us to apply Theorem~\ref{GNSW}.  
Importantly, the constants and the exceptional set obtained in Theorem~\ref{GNSW} are independent of $S$, which makes the argument work uniformly for all rational maps.

To extend this viewpoint to the analytic setting, we let $\mathcal E_f$ denote the set of all zeros of $f_i$, $0\le i\le n$, and define
\[
\mathcal Y_f := \mathbb C \setminus \mathcal E_f,
\]
where $f = (f_0,\ldots,f_n)$ is a holomorphic map such that the entire functions $f_i$, $0\le i\le n$, have no common zeros.  
Then $\mathcal Y_f$ is a parabolic open Riemann surface equipped with a smooth exhaustion function $\sigma$ for which the smooth $(1,1)$-form $dd^{c}\!\log \sigma$ has finite total mass on $\mathcal Y_f$.  
Although $\sigma$ may not be  a parabolic  exhaustion for   ${\mathcal Y}_f$, Chen et al.\ (see \cite{CHSX}) observed that, after appropriate modifications, the key ingredients of Nevanlinna theory still admit analogues in this setting.

\section{Nevanlinna theory and the GCD theorem on parabolic Riemann surfaces}\label{parabolic}

\subsection{Parabolic Riemann surfaces and Jensen's formula}\label{Jensen}    A
 non-compact Riemann surface ${\mathcal Y}$ is {\it parabolic} if it  admits  a smooth exhaustion function
 $\sigma: {\mathcal Y} \rightarrow [0, \infty)$
 such that $\log \sigma$ is harmonic outside a
 compact subset of ${\mathcal Y}$. Such $\sigma$ is called a {\it parabolic exhaustion}. Nevanlinna theory on parabolic open Riemann surfaces was recently developed by P${\rm{\breve{a}}}$un and Sibony \cite{PS}. 
Chen et al.  \cite{CHSX} considered the case when 
  ${\mathcal Y}:={\Bbb C}\backslash  \mathcal E$, where ${\mathcal E}:=\{a_j\}_{j=1}^{\infty}$  is a discrete countable set of points in ${\Bbb C}$.    The surface  ${\mathcal Y}$ is parabolic due to \cite[Lemma 10.1(g), p. 75]{Stoll} by taking $\phi: {\mathcal Y} \rightarrow {\Bbb C}$ as the
restriction of a transcendental function sending ${\mathcal E}$ to $0$. 
However  the smooth exhaustion function 
 $\sigma: {\mathcal Y} \rightarrow [0, \infty)$ constructed in \cite{CHSX},  which serves as our purpose,  is not necessarily compactly supported,  although the smooth (1,1)-form  $dd^c \log \sigma$ is of finite total mass on ${\mathcal Y}$.
  Therefore the parabolic Nevanlinna theory developed by  P${\rm{\breve{a}}}$un and Sibony is no longer directly applicable to  ${\mathcal Y}$.
     This difficulty arises because  the parabolic Jensen formula (see (\ref{jensen})) acquires an additional term 
$\int_{B_r^{\sigma}} \upsilon dd^c \tau$.   A key observation is that, by shrinking the radii $r_j$  sufficiently small, we can make (see Lemma \ref{Ch})
 $c^+_{g,\sigma}(r): =\int_{B_r^{\sigma}} \log^+ |g| dd^c \tau$
to be bounded (independent of $r$),   for any meromorphic function $g$ on ${\Bbb C}$.  The subtle point, however,  is that the required shrinking depends on the particular function 
$g$, so we must ensure in our arguments that such shrinking occurs only finitely many times.
 This ``shrinking trick"  was already observed by  Chen et al.  in  \cite{CHSX} and indeed plays a central role in their work. 
 Unfortunately  many of the necessary details were not made explicit there. 
We therefore supply  the details in this and the subsequent sections.
 
  Let us first briefly review the construction of $ \sigma$  given  in Lemma 3.1 and 3.2 of  \cite{CHSX}.  Let  
 ${\mathcal Y}:={\Bbb C}\backslash  \mathcal E$, where ${\mathcal E}:=\{a_j\}_{j=1}^{\infty}$  is a discrete countable set of points in ${\Bbb C}$. 
 Up
to translation and dilation on ${\Bbb C}$, we may assume that $|a_j| >2$ for all $j \ge  1$.   We  also assume that  the radii satisfy the following conditions:
(i) ${\mathbb D}(a_j,  2r_j )$ for all $j$ and the disc $ {\mathbb D}_2$ are disjoint,  and (ii) the sum $\sum_{j\ge 1} r_j <+\infty$. 

  The 
exhaustion function $\sigma$ on ${\mathcal Y}$ is obtained by smoothing the piecewise smooth  exhaustion function $\hat{\sigma}$  given by 
 \begin{equation}\label{X1} \log \hat{\sigma}:=\log^+|z|+\sum_{j\ge 1} r_j \log^+ {r_j\over |z-a_j|}.
 \end{equation}
  Note that $\hat{\tau}:= \log \hat{\sigma}$ takes values in $[0, +\infty)$, is continuous on ${\mathcal Y}$ and is smooth (indeed harmonic) outside the circle $S(0,1) :=\{|z| = 1\}$  and the disjoint circles
 $S(a_j, r_j):=\{|z-a_j|=r_j\}$. By the Poincar\'e-Lelong formula (see, e.g., \cite{ru2021nevanlinna}),  it is clear that 
$$dd^c\hat{\tau}={1\over 2} \nu(0, 1)+  {1\over 2}\sum_{j=1}^{\infty} r_j\left(\nu(a_j, r_j)-\delta_{a_j}\right)$$
is a distribution of order $0$ and locally of finite mass, where $\nu(a_j, r_j)$ is the  Haar measure on the circle $S(a_j, r_j)$.

  The smoothing process goes as follows (see  \cite[Appendix]{CHSX}):  Let 
$$\displaystyle h(r)=\begin{cases} 0 & \text{ for $r\leq {3\over 4}$}\\
 {1\over 4\pi} {1\over 1+e^{{r-1\over (r-1)^2-1/16}}}& \text{for ${3\over 4}<r<{5\over 4}$}\\
  {1\over 4\pi}  & \text{for $r\ge {5\over 4}$}\end{cases}
  $$ and 
 \begin{align}\label{H}
 H(r):=4\pi\int_0^r h(s-c){ds\over s},
 \end{align}
 where the constant  $c\in (0, 1/2)$ is chosen such that $H$ agrees with $\log^+r$ outside $[{1\over 2}, {3\over 2}]$. 
Note that $H$  is an absolute function independent of $a_j$. 
This function $H(r)$ is smooth of $\log^+r$  with $0\leq H(r)-\log^+r \leq \log^+{3\over 2}<{1\over 2}$ and $H(r) =\log^+r$ when $H(r)\ge \log{3\over 2}$. 
Define 
\begin{align}\label{tau}
\tau:=H(|z|)+\sum_{j=1}^{\infty} r_j H\left( {r_j\over |z-a_j|}\right), ~~~~\sigma=\exp(\tau).
\end{align}
Then,  as stated  in \cite[Lemma 3.1]{CHSX}, $\sigma \ge \hat{\sigma}$ and the difference 
$\sigma-\hat{\sigma}$ is supported on 
\begin{equation}\label{U}\text{Supp}(\sigma-\hat{\sigma})\subset U:=\left(A\left(0, {1\over 2}, {3\over 2}\right)\backslash {\mathcal E}\right)\bigcup \left(\bigcup_{j=1}^{\infty} A\left(a_j, {1\over 2}r_j, {3\over 2}r_j\right)\right),
\end{equation}
where $ A(a_j, {1\over 2}r_j, {3\over 2}r_j):=\{z\in {\mathcal Y}: {1\over 2}r_j\leq |z-a_j|\leq {3\over 2}r_j\}$ are pairwise disjoint annuli. For 
$z\not\in \cup_{j=1}^{\infty}\overline{{\mathbb D}(a_j, {3\over 2}r_j)}$, $\log^+\left|{r_j\over z-a_j}\right|=H\left(\left|{r_j\over z-a_j}\right|\right)=0$ for each $j$. Thus 
$\tau(z)=H(|z|)$ and $\hat{\tau}(z)=\log^+|z|$, which are equal when $\tau(z)\ge \log {3\over 2}$.  Hence $\sigma=\hat{\sigma}$ when $\sigma\ge {3\over 2}$.
 We say that {\it such $\sigma$ is associated to the discs $\{{\mathbb D}(a_j,  2r_j )\}$} by assuming that  $\{r_j\}$ satisfy (i) and (ii) above. Note that  $\sigma$ is not necessarily  a parabolic exhaustion function because  $dd^c\log \sigma$  may not be compactly supported. 
 
   Denote by $B^{\sigma}_r: =\{y\in {\mathcal Y}:\sigma(y) <r\}$ and by $S^{\sigma}_r:=\{y\in {\mathcal Y}:\sigma(y)=r\}$. 
 Denote by $\mu_r$ the measure induced by the differential $d^c\log\sigma|_{S^{\sigma}_r}$ and write $d\mu_r = d^c\log\sigma|_{S^{\sigma}_r}$. 
Note that, from  the above discussion (see also  \cite{CHSX}),
\begin{equation}\label{X} 
B_t^{\sigma}\subset {\mathbb D}_t~~~~\mbox{for}~t\ge 1,~~~\mbox{and} ~~~~~{\mathbb D}_t\backslash  \left(\bigcup_{j=1}^{\infty}{\mathbb D}(a_j, {3\over 2}r_j)\right)\subset B_t^{\sigma}~~~~\mbox{for}~t\ge {3\over 2}.
\end{equation}
  \begin{lemma}[Lemma 3.2 in \cite{CHSX}]
 The smooth 2-form  $dd^c\log \sigma$ defines an order  0 distribution of finite mass on ${\mathcal Y}$. 
 \end{lemma}
 \begin{proof} 
We use the formula: in polar coordinates $z = a + re^{i\theta}$, for some $a \in {\Bbb C}$, and for a smooth function $\phi$, one has 
$$d^c\phi={1\over 4\pi} \left(r{\partial \phi\over \partial r} d\theta -{1\over r}{\partial \phi\over \partial \theta} dr\right).$$
The smooth 2-form  $dd^c \log \sigma$ is supported on $U$,  where $U$ is given in (\ref{U}). 
On the annulus $A(0, {1\over 2}, {3\over 2})$, in polar coordinate $z = re^{i\theta}$,   by \eqref{H}  one has (see \cite[Appendix]{CHSX})
\begin{equation}\label{taubounedpart}
dd^c \tau=dd^cH(|z|)=d\left({r\over 4\pi} {\partial H(r)\over \partial r}d\theta\right) = d(h(r-c)d\theta)=h'\left(r-c\right)  dr\wedge d\theta=O(1)  dr\wedge d\theta.\end{equation}
Thus $dd^c \tau$ has finite mass on $A(0, {1\over 2}, {3\over 2})$. 
On the annulus  $A(a_j, {1\over 2}r_j, {3\over 2}r_j)$,      since $ {\mathbb D}(a_j,  2r_j )$ for all $j$ and the disc $ {\mathbb D}_2$ are disjoint, we have similarly that  $\tau=r_j H\left( {r_j\over |z-a_j|}\right)$ on $A(a_j, {1\over 2}r_j, {3\over 2}r_j)$. Therefore,
\begin{align}\label{tau_unbounedpart2} 
dd^c \tau&=dd^c \left(r_j H\left( {r_j\over |z-a_j|}\right)\right) \nonumber \\
&=
d\left({r\over 4\pi} {\partial r_j H\left({r_j\over r}\right)\over \partial r}d\theta\right)=\left({r_j\over r}\right)^2h'\left({r_j\over r}-c\right) dr\wedge d\theta=O(1)dr\wedge d\theta.
\end{align}
Since  $\sum_{j=1}^{\infty} r_j<+\infty$,  
 the support $U$ is of finite Lebesgue measure, this proves the lemma.
 \end{proof}
 
  We  now derive  the parabolic Jensen's  formula in our setting. For any function $\upsilon: {\mathcal Y}\rightarrow (-\infty, +\infty)$ such that  $dd^c\upsilon$ is of order 0,  the parabolic Jensen's formula (see \cite[Remark 3.2, p. 22]{PS}) is stated as  
\begin{equation}\label{jensen} \int_1^r {dt\over t} \int_{B_t^{\sigma}} dd^c \upsilon=\int_{S_r^{\sigma}} \upsilon d\mu_r - \int_{B_r^{\sigma}} \upsilon dd^c \tau  \quad (r>1).
\end{equation}
Let $g$ be   meromorphic  on ${\Bbb C}$. In (\ref{jensen}), by letting $\upsilon :=\log |g|$, and also using the Poincar\'e-Lelong formula, we get the following lemma.
 \begin{lemma}[Parabolic Jensen's formula]\label{jensenh} Let $g$ be a meromorphic  function on ${\Bbb C}$.  Then 
$$
\int_{S_r^{\sigma}} \log |g|  d\mu_r =N_{g, \sigma}(0, r)- N_{g, \sigma}(\infty, r)+  c_{g,\sigma}(r),
$$
where 
$N_{g,  \sigma}(0, r):=\int_1^r \sum_{z\in B_t^{\sigma} } \max\{0,{\rm ord}_zg\} {dt\over t},$ $ N_{g, \sigma}(\infty, r):=N_{\frac1g,  \sigma}(0, r)$, and 
\begin{equation}\label{ch}
 c_{g,\sigma}(r): =\int_{B_r^{\sigma}} \log |g| dd^c \tau.
\end{equation} 
\end{lemma}
Denote by 
 \begin{equation}\label{ch+}
  c^+_{g,\sigma}(r): =\int_{B_r^{\sigma}} \log^+ |g| dd^c \tau.
\end{equation}
 \begin{lemma}\label{Ch} Let $g$ be a meromorphic function on ${\Bbb C}$. 
Let ${\mathcal Y}\subset {\Bbb C}$ be given with   $\sigma$ being associated to the discs $\{{\mathbb D}(a_j,  2r_j )\}$.
   Then we can choose $r_j$ small enough such that  there exists a positive constant  $C_g$ depending on $g$ with
 $ c^+_{g,\sigma}(r)\leq C_g$ for any  $r>0$.  
 \end{lemma}

 \begin{proof} By using (\ref{X})  for $t\ge 2$,  and notice that  $dd^c \tau $ is supported on $U$  where $U$ is given in (\ref{U}),  
 we have 
\begin{align}\label{cgt}
 c^+_{g, \sigma}(t)&=\int_{B^{\sigma}_t}\log^+ |g| dd^c \tau\leq \int_{ {\mathbb D}_t \backslash \mathcal{E}}\log^+ |g| dd^c \tau\cr
& =\int_{A(0, {1\over 2}, {3\over 2}) }\log^+  |g| dd^c  \tau+ \int_{\cup_{j\ge 1}A(a_j, {1\over 2}r_j, {3\over 2}r_j)   \cap {\mathbb D}_t} \log^+ |g| dd^c \tau.
\end{align}
 Notice that $ \log^+ |g|$ is locally integrable around the zeros or poles  of $g$ and by (\ref{taubounedpart}), we get
$$
\int_{A(0, {1\over 2}, {3\over 2}) }\log^+  |g| dd^c  \tau<M_1,
$$ 
where $M_1$ is a positive  constant depending on $g$.

  We now estimate the integration over $A(a_j, {1\over 2}r_j, {3\over 2}r_j)$. 
 If $g(a_j)\not =\infty$. Then $\log^+ |g|\leq C_1$ for some $C_1>0$ on ${\mathbb D}(a_j, {3\over 2}r_j)$. 
Then, by (\ref{tau_unbounedpart2}),
 $$\int_{A(a_j, {1\over 2}r_j, {3\over 2}r_j)}\log^+  |g| dd^c \tau
 \leq C_1 \cdot \frac{9\pi}2 C_h r_j < {1\over 2^j},$$
by further shrinking $r_j$, where  $
C_h:=\max_{ \frac12\le s\le \frac32} |h'(s-c)|.$  If $g(a_j)=\infty$. Then, locally, $g(z)=(z-a_j)^{k_j} g_j(z)$ with $g_j(a_j)\not=0$. 
 Thus  $\log^+  |g|\leq C_2\log^+{1\over |z-a_j|}$ for some $C_2>0$ on  ${\mathbb D}(a_j, {3\over 2}r_j)$.
 We have similarly 
 $$ \int_{A(a_j, {1\over 2}r_j, {3\over 2}r_j)}\log^+|g| dd^c \tau \leq  
 \frac{9\pi}2 C_2 C_h\int_{{1\over 2}r_j}^{ {3\over 2}r_j}   \log rdr \le \frac{9\pi}2 C_2 C_h r_j^{\delta},
 $$
  for some $1<\delta<2$.
 Again,  we can choose $r_j$ sufficiently small  such that 
 $$\int_{A(a_j, {1\over 2}r_j, {3\over 2}r_j) }\log^+  |g| dd^c \tau< {1\over 2^j}.$$
Therefore, with this choice of the radii $r_j$,  we obtain from  \eqref{cgt} that
$  c^+_{g,\sigma}(r)\le M_1+\sum_{j=1}^{\infty} {1\over 2^j} =M_1+1$.
\end{proof} 
  Note that the shrinking process of $r_j$ in the above lemma depends on  the given function  $g$.  So we must ensure in our arguments that such shrinking occurs only finitely many times.  Throughout the rest of the paper, we will keep shrinking the radii $r_j$ whenever necessary.  The  $\sigma$ should be understood as being adjusted accordingly,
although we will continue to denote it by the same symbol.

\subsection{The characteristic functions  in the parabolic setting}  
 We continue to let 
${\mathcal Y}:={\Bbb C}\backslash  \mathcal E$ with $\sigma$  being the exhaustion associated to the discs $\{{\mathbb D}(a_j,  2r_j )\}$,  constructed in Section \ref{Jensen}. 
Let $f:{\Bbb C}\rightarrow  {\Bbb P}^n({\Bbb C})$ be a  holomorphic map. 
We briefly recall some standard  notations in Nevanlinna theory in the parabolic context (for details, see   \cite{CHSX}, \cite{PS}).

  Let $D=\{Q=0\}$ be a hypersurface in $\mathbb P^n({\Bbb C})$, where $Q\in\mathbb C[x_0,\hdots,x_n]$ is a homogeneous polynomial  of degree $d$. 
\begin{enumerate} 
\item[1.] The proximity function of $f$ with respect to $D$ is defined as $$m_{f, \sigma}(D, r)=\int_{S_r^{\sigma}}\lambda_D(f)  d\mu_r,$$
 where $\lambda_D$ is the Weil-function on ${\Bbb P}^n({\Bbb C})$ given by $$\lambda_D({\bf x})=\log {\|{\bf x}\|^d\over |Q({\bf x})|},~~{\bf x}=[x_0:\cdots:x_n]\in \mathbb P^n({\Bbb C}),  ~~~\|{\bf x}\|:= \max\{|x_0|,\dots,|x_n|\}.$$

\item[2.] The Nevanlinna characteristic (height)  function is defined by $$T_{f, \sigma}(r):=\int_{S_r^{\sigma}}\log \|{\bf f}\|d\mu_r, ~~~~r>1,$$
 where ${\bf f}=(f_0,\hdots,f_n)$ is a reduced representation of $f$.
 The   Ahlfors-Shimizu characteristic (height)  function is defined by
 $$T^{A}_{f, \sigma}(r)= \int_1^r {dt\over t} \int_{B_t^{\sigma}} f^*\omega_{FS}, ~~~~r>1,$$
 where $\omega_{FS}$ denotes the Fubini-Study form.
 Note that $T_{f, \sigma}(r)$ depends on the choices of the representation ${\bf f}$ of $f$, while $T^{A}_{f, \sigma}(r)$ does not depend on  the representation.
 Therefore, when we write $T_{f, \sigma}(r)$, we always assume that we have fixed  a reduced a representation ${\bf f}=(f_0, \dots, f_n)$ of $f$.

 \item[3.] The counting function and the $k$-truncated counting function are defined by 
 $$N_{f,  \sigma}(D, r):=\int_1^r \sum_{z\in B_t^{\sigma}} {\rm ord}_zf^*D{dt\over t}, ~~~~~~~\mbox{and}~~N_{f, \sigma}^{(k)}(D, r):=\int_1^r \sum_{z\in B_t^{\sigma}} \min\{k, {\rm ord}_zf^*D\}{dt\over t}.$$

\item[4.] The counting function extends naturally to the sub-schemes, as well as the notion  $N_{\text{gcd}, \sigma}(f, g)(r)$, where $f,g$ are meromorphic functions on $\mathcal Y$.  See  \cite{levin2019greatest}.
\end{enumerate} 
 
 We also use the standard notation $T_f(r), m_f(D, r), n_f(D, t)$ and $N_f(D, r)$ of the Nevanlinna functions associated to a holomorphic map $f:  {\Bbb C}\rightarrow {\Bbb P}^n({\Bbb C})$.
In this case, the Ahlfors-Shimizu  and the classical Nevanlinna characteristic function agree up to an additive constant (see \cite[Theorem 1.1.19]{noguchibook}), i.e., 
$
T_f(r)-T^A_f(r)=O(1).$

\begin{theorem}[The parabolic  First Main Theorem]\label{FMT} Let ${\mathcal Y}\subset {\Bbb C}$ be as above.
Let $D=\{Q=0\}$ be a hypersurface in $\mathbb P^n({\Bbb C})$, where $Q\in\mathbb C[x_0,\hdots,x_n]$ is a homogeneous polynomial  of degree $d$.  
Let $f =(f_0,\hdots,f_n):  {\Bbb C}\rightarrow {\Bbb P}^n({\Bbb C})$ be a holomorphic curve, where $f_0,\hdots,f_n$ are entire functions without common zero.    Suppose that $ {f}({\mathcal Y})\not\subset \mbox{Supp}(D)$. 
Then
$$dT_{f, \sigma}(r) = m_{f, \sigma}(D, r)+N_{f, \sigma}(D, r) + c_{Q({\bf f}), \sigma}(r).$$
\end{theorem}}
\begin{proof}
 Following the definition of Nevanlinna characteristic function, we have 
$$dT_{f, \sigma}(r)=\int_{S_r^{\sigma}}\log \|{\bf f}\|^dd\mu_r = \int_{S_r^{\sigma}}\log {\|{\bf f}\|^d \over |Q({\bf f})|}d\mu_r + \int_{S_r^{\sigma}}\log {|Q({\bf f})|}d\mu_r.$$
Hence the assertion follows from  Lemma \ref{jensenh} together with the definition of the proximity function.
\end{proof}

 In the case of a meromorphic function $g$ on ${\mathcal Y}$, write $g=g_0/g_1$, where $g_0$ and $g_1$ are entire functions without common zeros.
 We regard  the function $g$ as the map $g: = [g_0 : g_1] :   {\mathcal Y}\rightarrow {\Bbb P}^1({\Bbb C})$.  By taking $Q(z_0, z_1)=z_1$ (so  $D= \{z_1=0\}=\infty$), we get 
$$m_{g,  \sigma}(r):=m_{g, \sigma}(\infty, r)=\int_{S_r^{\sigma}} \log {\max\{|g_0|,|g_1|\}\over |g_1|} d\mu_r = \int_{S_r^{\sigma}} \log^+|g| d\mu_r +O(1).$$
The above $m_{g,  \sigma}(r)$ is well-defined, i.e., it is independent of  the choice of $g_0$ and $g_1$.
However, the Nevanlinna height function  $T_{g, \sigma}(r)= \int_{S_r^{\sigma}} \log \max\{|g_0|, |g_1|\} d\mu_r$ is not well-defined.  Instead  we introduce a new  height function that is well-defined: 
 \begin{equation}\label{Charcmero} 
T^{M}_{g, \sigma}(r) = m_{g, \sigma}(\infty, r)+N_{g, \sigma}(\infty, r).
\end{equation} 
From the parabolic Jenen's formula above, 
\begin{equation}\label{rci}
T^{M}_{1/g, \sigma} (r)= T^{M}_{g, \sigma}(r) +c_{1/g, \sigma}(r).
\end{equation}

We now compare $T^{M}_{g, \sigma} (r)$ with $T^A_{g, \sigma} (r)$.
From parabolic Jensen's formula, 
\begin{eqnarray*}
T^A_{g, \sigma}(r) &=&\int_1^r {dt\over t} \int_{B_t^{\sigma}} dd^c   \log \max\{|g_0|,  |g_1|\}=\int_{S_r^{\sigma}}  \log \max\{|g_0|,  |g_1|\} d\mu_r  \\
&~& -  \int_{B_r^{\sigma}} \log \max\{|g_0|,  |g_1|\} dd^c \tau \\
&=& \int_{S_r^{\sigma}} \log^+ |g| d\mu_r +  \int_{S_r^{\sigma}} \log |g_1| d\mu_r   -  \int_{B_r^{\sigma}} \log \max\{|g_0|,  |g_1|\} dd^c \tau \\
 &=& T^{M}_{g, \sigma}(r) + {1\over 2} \int_{B_r^{\sigma}} \log \left(1/(1+|g|^2)\right) dd^c \tau = T^{M}_{g, \sigma}(r) + {1\over 2}c_{1/(1+|g|^2), \sigma}(r).
 \end{eqnarray*}
Similarly to the proof  of Lemma~\ref{Ch}, by  choosing  $r_j$ small enough,  there exists a constant $C>0$ with  $c^+_{1/(1+|g|^2), \sigma}(r)\leq C$. Hence
\begin{equation}\label{Amero}
|T^M_{g, \sigma}(r)-T^A_{g, \sigma}(r)|\leq C.
\end{equation}
It is important to note that these radii $r_j$ depend on the function $g$.
 \begin{proposition}[\cite{CHSX}]\label{exhuastion} 
Let  $f =(f_0,\hdots,f_n):  {\Bbb C}\rightarrow {\Bbb P}^n({\Bbb C})$ be a holomorphic curve, where $f_0,\hdots,f_n$ are entire functions without common zero.  Let ${\mathcal Y}\subset {\Bbb C}$ be given with   $\sigma$ being associated to the discs $\{{\mathbb D}(a_j,  2r_j )\}$.
Then we  can choose $r_j$ small enough such that,  for all $r\ge 1$,  the following hold: 
\begin{enumerate}
\item $O(1)\le  T_f(r)-T^A_{f, \sigma}(r)\le \log r +O(1).$  

\item $O(1)\le  T_{f}(r)-T_{f, \sigma}(r)\le \log r + O(1).$  
\end{enumerate}
 \end{proposition}
  \begin{proof}   From \eqref{X},  for $t\ge {3\over 2}$, $B_t^{\sigma}\subset {\mathbb D}_t$  and 
 $${\mathbb D}_t\backslash \left(\bigcup_{j=1}^{\infty}D(a_j, {3\over 2}r_j)\right)\subset B_t^{\sigma}.$$  Thus
 \begin{align*} & 0\leq T_f^A(r)-T^A_{f, \sigma}(r) \leq \int_{{3\over 2}}^r {dt\over t} \int_{ \left(\bigcup_{j=1}^{\infty}{\mathbb D}(a_j, {3\over 2}r_j)\right)\cap {\mathbb D}_t} f^*\omega_{FS}+O(1)\cr
 &\leq  \int_1^r {dt\over t} \int_{\bigcup_{j=1}^{\infty}{\mathbb D}(a_j, {3\over 2}r_j)} f^*\omega_{FS} +O(1)\cr
 &\leq  \int_1^r {dt\over t}\left(\sum_{j=1}^{\infty}  \int_{{\mathbb D}(a_j, {3\over 2}r_j)} f^*\omega_{FS}\right) +O(1), ~~~~(r\ge 1).
 \end{align*}
We
choose $r_j > 0$ sufficiently small so that \begin{equation}\label{Rur} \int_{{\mathbb D}(a_j, {3\over 2}r_j)} f^*\omega_{FS} <2^{-j}.\end{equation}
Hence the above estimate yields $$0\le  T_f^A(r)-T^A_{f, \sigma}(r)\le \log r +O(1), ~~~~(r\ge 1).$$
By using the fact that $T_f(r)=T_f^A(r)+O(1)$, 
this completes the proof of (i).

 To prove (ii), 
by using (\ref{jensen}) with $\upsilon:=\log\|{\bf f}\|$, we get 
$$T^A_{f, \sigma}(r)=
 \int_1^r {dt\over t} \int_{B_t^{\sigma}} dd^c \log \|{\bf f}\|
 =\int_{S_r^{\sigma}}   \log \|{\bf f}\| d\mu_r - \int_{B_r^{\sigma}}  \log \|{\bf f}\|dd^c \tau
= T_{f, \sigma}(r) - \int_{B_r^{\sigma}}  \log \|{\bf f}\| dd^c \tau.
 $$
 Thus $$|T^A_{f, \sigma}(r)-  T_{f, \sigma}(r)|\leq \int_{B_r^{\sigma}}  \log^+ \|{\bf f}\| dd^c \tau + \int_{B_r^{\sigma}}  \log^+ (1/\|{\bf f}\|)dd^c \tau.$$
 Since $f_0,\hdots,f_n$ have no common zero, we can argue similarly to the proof  of Lemma~\ref{Ch} to choose $r_j$ small enough so that there exists a constant $C>0$ with  
 $$\int_{B_r^{\sigma}}  \log^+ \|{\bf f}\| dd^c \tau+ \int_{B_r^{\sigma}}  \log^+ \frac 1{\|{\bf f}\|}  dd^c \tau \le C.$$  for any  $r>0$. 
By combining above with (i), we prove (ii). 
\end{proof} Again,  these radii $r_j$ in Proposition \ref{exhuastion}  depends on the map $f$.

\subsection{The parabolic Logarithmic Derivative Lemma and H. Cartan's theorem}  
 To establish the parabolic Logarithmic Derivative Lemma, we introduce the  weighted Euler characteristic function (see \cite[Definition 3.4]{PS}):
 $$ \mathfrak{X}^+_{\sigma}(r):=\int_{S^{\sigma}_r} \log^+ \left|d\sigma\left({\partial \over \partial z}\right)\right|^2d\mu_r.$$ 
Using the properties that the disks ${\Bbb D}(a_j ,2r_j )$ are disjoint and that $\sum_{j\ge 1} r_j < +\infty$, we see from 
\cite[ Example (2), pp. 32 and 33]{PS}  (see also \cite[(2.10)]{CHSX}) that
$$ \mathfrak{X}^+_{\sigma}(r) =  \mathfrak{X}_{\sigma}(r) +O(\log r)=N_{\mathcal E}(r)+O(\log r), \quad r>1,$$
where $N_{\mathcal E}(r):=\int_1^r n_{\mathcal E}(t){dt\over t} $, and $n_{\mathcal E}(t)$ counts the number of points in ${\mathcal E}\cap \mathbb{D}_t$ for $t\ge 1$.  In particular, 
if  ${\mathcal E}: = f^{-1}(D)$ for a holomorphic map $f: {\Bbb C}\rightarrow {\Bbb P}^n({\Bbb C})$ and a divisor  $D$  on ${\Bbb P}^n({\Bbb C})$
 with $f({\Bbb C})\not\subset \mbox{Supp} D$,   it follows (see also \cite[(4.2)]{CHSX}) that 
\begin{equation}\label{sigmaestimae}  
\mathfrak{X}^+_{\sigma}(r)\leq N_f^{(1)}(D, r) + O(\log r),\quad r>1.
\end{equation}
The above estimate is crucial to our paper.  
In \cite{CHSX}, the authors consider   the case  when $N_f(D, r)=o(T_f(r))$, and therefore  $\mathfrak{X}^+_{\sigma}(r)$ is a small function with respect to $T_f(r)$.  In our case, we treat the situation where $N_f^{(1)}(D, r)\leq {1\over \ell} T_f(r)$ for a sufficiently large $\ell$.   In this setting, $\mathfrak{X}^+_{\sigma}(r)$  should be regarded as having the same order of growth as $T_f(r)$.

\begin{lemma}[Parabolic Logarithmic Derivative Lemma]\label{DL} Let $g$ be a meromorphic function on $\mathbb C$. Let ${\mathcal Y}\subset {\Bbb C}$ be as above. For any $\delta> 0$, one has
\begin{enumerate}
\item $m_{g'/g, \sigma}(r) \le_{\rm exc} {1\over 2} ((1+\delta)^2 \log  T_{g}(r)  + \delta\log r+ \mathfrak{X}^+_{\sigma}(r)).$
\item $m_{g^{(l)}/g, \sigma}(r) \le_{\rm exc} {l\over 2} \left((1+\delta)^2 \log T_g(r)  + \delta\log r+ \mathfrak{X}^+_{\sigma}(r)\right).$
\end{enumerate}
\end{lemma}
\begin{proof}  (i) We follow the the approach in \cite[Theorem 3.8, pp. 26-28]{PS}. Consider the auxiliary form
$$\Omega:= {1\over |w|^2(1+\log^2 |w|)}{\sqrt{-1}\over 2\pi} dw\wedge d\bar{w},$$
introduced in \cite[(15), p. 27]{PS}, which has finite volume on ${\Bbb C}$.
By the change of variable formula, we have 
 $$\int_1^r{dt\over t} \int_{B_t^{\sigma}} g^*\Omega= \int_{w\in {\Bbb C}}N_{g, \sigma}(w, r)\Omega.$$
 Since $B_t^{\sigma}\subset \mathbb{D}_t$ for any $t\ge 1$, and by the classical First Main Theorem, 
$$N_{g, \sigma}(w, r)\leq N_g(w, r)  \leq T_g (r) + O(1), ~~~\mbox{for}~r>1.$$
Thus 
$$\int_1^r{dt\over t} \int_{B_t^{\sigma}} g^*\Omega= \int_{w\in {\Bbb C}}N_{g, \sigma}(w, r)\Omega \leq \int_{\Bbb C} T_g(r)\Omega+\int_{\Bbb C} C\Omega.$$
 Consequently, we obtain
$$\int_1^r{dt\over t} \int_{B_t^{\sigma}} g^*\Omega\leq C_0 T_g(r)$$
for some constant $C_0 >0$.  The rest of the proof follows from \cite[Theorem 3.8, p. 28] {PS}.
This finishes the proof of (i).

 To prove (ii),  notice that
$${g^{(l)} \over g} ={g^{(l)} \over g^{(l-1)}}{g^{(l-1)} \over g^{(l-2)}}\cdots {g' \over g}.$$
Therefore, by (i), we obtain
$$m_{g^{(l)}/g, \sigma}(r) \le \sum_{j=1}^l m_{g^{(j)}/g^{(j-1)}, \sigma}(r) \le_{\rm exc} {1\over 2}\left((1+\delta)^2 \sum_{j=1}^l \log T_{g^{(j)}}(r)  + \delta l \log r+ l\mathfrak{X}^+_{\sigma}(r)\right).$$
By the classical First Main Theorem, 
$$
T_{g^{(j)}}(r)= m_{g^{(j)}}(r)+ N_{g^{(j)}}(r) \le m_{g^{(j)}/g^{(j-1)}}(r)+  m_{g^{(j-1)}}(r)+ 2N_{g^{(j-1)}}(r)
\leq_{\rm exc} 3T _{g^{(j-1)}}(r)+O(1).$$
Hence 
$$T_{g^{(j)}}(r)\leq_{\rm exc}  3^j T_g(r) +O(1).$$
Therefore, we get
$$m_{g^{(l)}/g, \sigma}(r) \le_{\rm exc} {l\over 2} \left((1+\delta)^2 \log T_g(r)  + \delta\log r+ \mathfrak{X}^+_{\sigma}(r)\right).$$
\end{proof}
In particular, by the definition,  the above lemma gives the following: 
\begin{lemma}\label{unit} If $u$ is an entire function such that it has no zeros on $\mathcal Y$ $($i.e. $u$ is a unit on $\mathcal Y$$)$, then, for any $\delta> 0$, 
$$T^{M}_{u'/u, \sigma}(r) \le_{\rm exc} {1\over 2}((1+\delta)^2 \log T_u(r)  + \delta \log r+ \mathfrak{X}^+_{\sigma}(r)).$$
\end{lemma}
 
 Using the parabolic Logarithmic Derivative Lemma and the parabolic Jensen's formula, we can easily derive the following parabolic version of H. Cartan's theorem.  
\begin{theorem}\label{cartan} Let  ${\mathcal Y}:={\Bbb C}\backslash  \mathcal E$ with the  exhaustion function $\sigma$ constructed above.   Let  $f= ( f_0, \dots, f_n): {\Bbb C}  \rightarrow {\mathbb P}^n({\mathbb C})$ be a 
	holomorphic map whose image is not contained in any proper linear subspace  of ${\mathbb P}^n({\mathbb C})$, where $f_0, \dots, f_n$ are 
	entire functions  without common zeros.  Let $H_1,\dots,H_q$ be  hyperplanes on $\mathbb{P}^n(\mathbb{C})$.
	 Then, for $\delta >0$, 
	\begin{align*}
	&\int_{S_r^{\sigma}} \max_{ J}\sum_{j\in J} \lambda_{H_j}(f)  d\mu_r+ N_{W, \sigma}(0, r)\\
	&\leq_{\rm exc}
	(n+1)T_{f, \sigma}(r) + O(\log T_f(r))+  \frac{n(n+1)}2(\mathfrak{X}^+_{\sigma}(r)+\delta \log r) +c_{\frac1{W}, \sigma}(r),
	\end{align*}
	where  $W :=W( f_0, \dots, f_n)$ is the Wronskian of $f$ and  the maximum is taken over all subsets $J\subset \{1, 2, \dots, q\}$ such that  $H_j, j\in J$,  are in general position. 	
	\end{theorem}
\begin{proof} 
We follow the  proof in \cite[Theorem A.5.1.3]{ru2021nevanlinna}. 
 Let $H_j=\{{\bf w}\in {\Bbb P}^n({\Bbb C})~|~<{{\bf w}, \bf a}_j>=0\}$ for $j=1, \dots, q$.
Let $T$ be the set of all 
injective maps $\mu: \{0,1,\dots,n\} \rightarrow \{1,\dots,q\}$ such that 
${\bf a}_{\mu(0)}, \dots, {\bf a}_{\mu(n)}$ are linearly independent.
Then 
\begin{eqnarray}\label{512}
\int_{S_r^{\sigma}}\max_{J}\sum_{j \in J}
\lambda_{H_j}(f) d\mu_r &\leq & \int_{S_r^{\sigma}}\max_{\mu \in T}\sum_{j=0}^n \log~\left({\|{\bf f}\|\|{\bf a}_{\mu(j)}\|\over |<{\bf f}, {\bf a}>|}\right) d\mu_r  \nonumber\\ 
 &\leq &  \int_{S_r^{\sigma}} \max_{\mu \in T} \log^+{|W(<{\bf f}, {\bf a}_{\mu(0)}>, \dots, <{\bf f}, {\bf a}_{\mu(n)}>)|\over 
\prod_{j=0}^n|<{\bf f},  {\bf a}_{\mu(j)}>|}  d\mu_r  \nonumber\\
&+& (n+1) T_{f, \sigma}(r) + 
\int_{S_r^{\sigma}}
\log (1/| W( f_0, \dots, f_n)|)  d\mu_r + O(1).
\end{eqnarray}
By applying the logarithmic derivative lemma (Lemma \ref{DL}) and the parabolic Jensen formula (Lemma \ref{jensenh}),  the theorem follows.
\end{proof}

The following result can be derived from the proof of \cite[Theorem 2.1]{RuWang2003} and Lemma \ref{DL}.
\begin{theorem}\label{trunborel}  Let $f_0,\dots,f_n, f_{n+1}$ be  entire functions without common  zeros.  Let  ${\mathcal Y}:={\Bbb C}\backslash  \mathcal E$ with the  exhaustion function $\sigma$ constructed above.  
	Assume that $f_0+\dots+ f_n+ f_{n+1}=0.$ 
	If $\sum_{i\in I} f_i\ne 0$ for any proper subset $I\subset\{0,\dots,n+1\}$, then we have 
	\begin{equation*}
	  T^{M}_{f_i/f_j, \sigma}(r)   \leq_{\rm exc}  \sum_{l=0}^{n+1} N_{f_l,\sigma}^{(n )}(0,r)+\frac{n(n+1)}2  (\mathfrak{X}^+_{\sigma}(r)+  \log r) +c_{\frac1{D_0},\sigma}(r)+ \sum_{\substack{l=0 \\ l\neq j}}^{n+1} c_{f_l, \sigma}(r)+O\left(\max_{1\leq l\leq n}\log T_{\frac{f_l}{f_0}} (r)\right),
 	\end{equation*}
where   $D_0$ is an entire function appeared in the proof of \cite[Theorem 2.1]{RuWang2003} or \cite[Theorem A5.2.6]{ru2021nevanlinna}), which is determined in terms of $f_0, \dots, f_n$ only.
\end{theorem}  
\begin{proof}  
We first sketch the proof by assuming that $f_0, \dots, f_n$ are linearly independent.
By applying  Theorem \ref{cartan} to  $f=(f_0, \dots, f_{n+1})$, we have
$$\sum_{j=1}^{n+2}  m_{f, \sigma}(H_j, r)+N_{W, \sigma}(r) \le (n+1)T_{f, \sigma}(r) + \frac{n(n+1)}2  (\mathfrak{X}^+_{\sigma}(r)+\delta \log r) +c_{\frac1{W},\sigma}(r)+ O(\log T_{f}(r)),
$$
where $H_i=\{x_{i-1}=0\}, i=1, \dots, n+1$ and $H_{n+2}=\{x_0+\cdots +x_n=0\}$.
By Theorem \ref{FMT},    for $1 \leq j\leq n+2$, 
$$T_{f, \sigma}(r) = m_{f, \sigma}(H_j, r)+ N_{f, \sigma}(H_j, r)+ c_{f_{j-1}, \sigma}(r).$$
Thus we get,
$$T_{f, \sigma}(r)  \leq   \sum_{j=0}^{n+1}  N_{f_j, \sigma}(0, r)- N_{W, \sigma}(r)+\frac{n(n+1)}2  (\mathfrak{X}^+_{\sigma}(r)+\delta \log r) +c_{\frac1{W},\sigma}(r)+ \sum_{j=0}^{n+1} c_{f_j, \sigma}(r)+ O(\log T_{f}(r)).$$
It is standard to verify that 
$$
\sum_{j=0}^{n+1}  N_{f_j, \sigma}(0, r)- N_{W, \sigma}(r)\le \sum_{j=0}^{n+1}  N^{(n)}_{f_j, \sigma}(0, r).
$$
The proof can be completed by applying the following estimate, valid for any for any $0\le i\le n$ and $0\le j\le n+1$:
\begin{align}
T_{f, \sigma}(r) &=\int_{S_r^{\sigma}} \log \max_{0\leq l\leq n+1} |f_l(z)| d\mu_r= \int_{S_r^{\sigma}} \log \max_{0\leq l\leq n+1} (|f_l(z)|/|f_j(z)|)d\mu_r+\int_{S_r^{\sigma}} \log |f_j(z)| d\mu_r\nonumber \\
&\ge \int_{S_r^{\sigma}} \log^+ \left |\frac{f_i(z)}{f_j(z)}\right| d\mu_r +N_{f_j, \sigma}(0, r) +  c_{f_j,\sigma}(r)\nonumber \\
&\ge m_{\frac{f_i }{f_j },\sigma}(r)+N_{\frac{f_i }{f_j }, \sigma}(\infty, r)+ c_{f_j,\sigma}(r)=T^{M}_{f_i/f_j, \sigma}(r)+  c_{f_j,\sigma}(r).
\end{align}
  Thus 
$$  T^{M}_{f_i/f_j,\sigma}(r)   \leq  \sum_{j=0}^{n+1}  N^{(n)}_{f_j, \sigma}(0, r)+ \frac{n(n+1)}2  (\mathfrak{X}^+_{\sigma}(r)+\delta \log r) +c_{\frac1{W},\sigma}(r)+ \sum_{\substack{l=0 \\ l\neq j}}^{n+1} c_{f_l, \sigma}(r)+ O\left(\max_{1\leq l\leq n}\log T_{\frac{f_l}{f_0}} (r)\right).$$

 Passing from the condition that  $f_0, \dots, f_n$ are linearly independent to the situation in which no proper subsum of $f_0+\dots+f_{n+1}$ vanishes involves only algebraic arguments. In particular, the resulting statements are obtained from the present ones by changing $W$ to $D_0$, the modified Wronskian that appears in the proof of \cite[Theorem 2.1]{RuWang2003}.
\end{proof}


\subsection{The moving GCD theorem in the parabolic setting}
One of the key components in \cite{GW22} is the extension of the GCD theorem in \cite{levin2019greatest} to the case where the coefficients of two coprime polynomials are allowed to be meromorphic functions of the same growth. In this subsection, we aim to extend this line of argument to the parabolic setting. Before we proceed, let us briefly outline the development of the GCD theorem.

  The foundational arguments in \cite{levin2019greatest} are based on results initially established in \cite{levin_gcd} for the number field case. In \cite{WaYa}, the GCD results concerning  $\mathbb{G}_m^n$ from \cite{levin_gcd} were generalized to more general varieties by combining the approach in \cite{Silverman2005} with Diophantine approximation results for arbitrary varieties introduced by the first named author and Vojta \cite{ruvojta}. More recently, a further generalization was achieved in \cite{HL}, which applied Vojta's results on integral points on semiabelian varieties for the number field case, as well as \cite{NWY02} for the complex case.

  To proceed, we will first derive a version of the second main theorem with moving targets of the same growth, which will serve our purpose. See \cite{RuStoll1991} or \cite[Chapter 6]{ru2021nevanlinna}.

  Let $a_{0}, \hdots,a_{n}$ be meromorphic functions, not all zeros, on ${\Bbb C}$, and let $L:=a_{0} X_0+\dots+a_{n} X_n$.  
Then  $L$ defines a  hyperplane $H$ in $\PP^n $, over the field  generated by  $a_{0}, \hdots,a_{n}$ over $\mathbb C$.
We note that  $H(z)$ is the hyperplane determined by the linear form $L(z)$. The definition of the Weil function, proximity function and counting function can be easily extended to moving hyperplanes.  
For example,
\begin{align}\label{Weil}
	\lambda_{H(z)}(P)=-\log \frac{| a_{0} (z)x_0+\cdots+ a_{n} (z)x_n|}{\max\{|x_0|,\dots ,|x_n|\}\max\{| a_{0} (z)|,\dots ,| a_{n} (z)|\}},
\end{align}
for $z\in\mathcal Y$ which is not a common zero of $a_{0}, \cdots,a_{n}$, or a pole of any $a_{k}$, $0\le k\le n$.
   We now consider a family of linear forms $L_j:=\sum_{i=0}^n a_{ij}X_j$, $1\le j\le q$, as the aformentioned linear forms.   Without loss of generality, we normalize the linear forms $L_j$, $1\le j\le q$,  such that for each $1\le j\le q$, there exists $0\le j'\le n$ with   $a_{jj'}=1$.    Let $t$ be a positive integer and let $V(t)$ be the complex vector space spanned by the elements
\begin{align*}
	\left\{ \prod a_{jk}^{n_{jk}} : n_{jk}\ge 0,\,\sum n_{jk}=t \right\},
\end{align*}
where the products and sums run over $1\le j\le q$ and $0\le k\le n$.
Let $1=b_1, b_2, \cdots,b_u$ be a basis of $V(t)$ and $b_1,\cdots,b_w$ a basis of $V(t+1)$.  It's clear that $u\le w$.
Moreover, we have 
\begin{align}\label{vlimit}
	\liminf_{t\to\infty}\dim V(t+1)/\dim V(t)=1. 
\end{align}
  See \cite[Lemma 6]{WaRoth} for a proof.

\begin{definition}\label{degenerateV}
	Let $V$ be a $\mathbb C$-vector space spanned by finitely many meromorphic functions.  We say that a holomorphic map $f=(f_0,\dots,f_n): {\Bbb C}\to\mathbb P^n({\Bbb C})$ is linearly   independent  over $V$ if whenever we have a linear combination 
$  \sum_{i=0}^{n}a_{i}f_{i}=0$ with $a_{i}\in V$, then $a_{i}=0$ for each $i$.	Otherwise we say
	that $f$ is  linearly dependent  over $V$. 
\end{definition}

The following formulation of the second main theorem with moving targets  follows from the strategy in \cite{Wang2000}  and the proof of  \cite[Theorem A6.2.4]{{ru2021nevanlinna}}, by adding the Wronskian term and computing the error term explicitly when applying the second main theorem (Theorem \ref{cartan}).	 
\begin{theorem}
	\label{movingsmt0}
	Let $f=(f_0,\hdots,f_n): {\Bbb C}\to\mathbb{P}^n(\mathbb{C})$ be a holomorphic map  where $f_0,f_1,\hdots,f_n$ are  entire functions without common zeros.  Let $H_j$, $1\le j\le q$, be arbitrary (moving) hyperplanes given by $L_j:=a_{j0} X_0+\dots+a_{jn} X_n$ where $a_{j0}, \cdots,a_{jn}$ are   meromorphic functions, not all zeros on ${\Bbb C}$.  Let $t$ be a positive integer,    $1=b_1, b_2, \cdots,b_u$ be a basis of $V(t)$ and $b_1,\cdots,b_w$ a basis of $V(t+1)$. Denote by $W$ the Wronskian of $\{hb_mf_k\,|\, 1\le m\le w,\,  0\le k\le n\}$, where $h$ is a meromorphic function such that $hb_1,\hdots, hb_w$ are entire functions without common zeros. Let  $\mathcal Y\subset {\Bbb C}$ be as above.  If $f$ is linearly non-degenerate over $ V(t+1) $, then, we have the following inequality
	\begin{align}\label{MSMTpre} \int_{S^{\sigma}_r}\max_J \sum_{k\in J}&\lambda_{H_k(z)}(f(z))d\mu_r+\frac 1uN_{W, \sigma}(0,r)\cr
		&\leq_{\operatorname{exc}}   \left(\frac wu (n+1)\right)T_{f, \sigma}(r)+\frac wu (n+1)(t+2)  \sum_{j=1}^q T_{L_{j}, \sigma}(r) \cr
		 &+\frac w{2u} (n+1)(nw+w-1)(\mathfrak{X}^+_{\sigma}(r)+\log r) + \frac 1u c_{{1\over W}, \sigma}(r)+  O(\log T_f(r)), 
	\end{align}  where  $ T_{L_{j}, \sigma}(r) =  \sum_{i=0}^n T^M_{a_{ji},\sigma}(r)$,   $w=\dim V(t+1)$, $u=\dim V(t)$, and the maximum is taken over all subsets $J$ of $\{1,\dots, q\}$ such that $H_j(z)$, $j\in J$, are in general position. 
\end{theorem}

\begin{remark*}
Let $\epsilon$ be a sufficiently small positive number. 
By \eqref{vlimit}, we find $t$ such that $\frac wu\le 1+\epsilon$.  An explicit upper bound for $t$ can be found in the proof of \cite[Proposition 5.3]{CDMZ} in term of $\epsilon$ and the number of generator for $V(1)$, which is less than $qn$.  More precisely, if $\epsilon<\frac1{16}$, then we can find 
\begin{align}\label{boundt}
t\le \frac{3qn}{\epsilon}\log\frac1\epsilon
\end{align}
such that $\frac wu\le 1+\epsilon$. 
Furthermore,  for such $t$, we have
\begin{align}\label{dimension}
	 u:=\dim V(t)\le \left(\frac 1{\epsilon}\right)^{qn}, ~~~\mbox{and consequently}, ~~~~
	 w:=\dim V(t+1)\le (1+\epsilon)\left(\frac 1{\epsilon}\right)^{qn}. 
\end{align}
\end{remark*}

Using the above theorem, and through the same argument from the proofs of \cite[Theorem 5.7]{levin2019greatest},  \cite[Theorem 26]{GSW}, and   \cite[Theorem 4.4]{GW22}, we can get the 
following GCD theorem (note we now consider the case that  $u_1,\hdots,u_n$ are units on $\mathcal Y$).  We first fix some notation. Let $K$ be a subfield of the field of meromorphic functions on ${\Bbb C}$.
	Let $F,G $ be coprime  polynomials in $n$ variables of the same degree $d>0$  over $K$.  Assume that one of the coefficients in the expansions of $F$ and $G$ is equal to 1.
		For every positive integer $t$,
	we denote by $V_{F,G}(t)$ the (finite-dimensional) $\mathbb C$-vector
	space  spanned by $\prod_{\alpha}\alpha^{n_{\alpha}}$,
	where $\alpha$ runs through all non-zero coefficients of $F$ and $G$, $n_{\alpha}\ge0$ and $\sum n_{\alpha}=t$; we also put
	$d_{t}:=\dim V_{F,G}(t)$.  
	For every  integer  $m\ge d$, we let $M=M_{m,n,d}=2\binom {m+n-d}{n}- \binom {m+n-2d}{n}$.  
   Let
$Q =\sum_{{\mathbf i}\in I_Q } a_{\mathbf i}\cdot{\mathbf x}^{\mathbf i}\in K[x_1,\hdots,x_n]$, where ${\mathbf i}=(i_1,\hdots,i_n)$, ${\mathbf x}^{\mathbf i}=x_1^{i_1}\hdots x_n^{i_n}$, and $a_{\mathbf i}\ne 0$ if ${\mathbf i}\in I_Q $.  
We define  the parabolic characteristic function of $Q$ as 
\begin{align}\label{ParacharF}
	T_{Q, \sigma }(r):=\sum_{{\mathbf i}\in I_Q }T^M_{a_{\mathbf i},\sigma}(r).
\end{align}
	\begin{theorem}[GCD theorem on parabolic  Riemann surfaces evaluating at  units]\label{Mfundamental}
	Let $n\ge 2$ be a positive integer.  Let $K$ be a subfield of the field of meromorphic functions on ${\Bbb C}$. Let  $\mathcal Y\subset {\Bbb C}$ be as above. 
	Let ${\bf u}=(u_1,\hdots,u_n)$, where  $u_1,\hdots,u_n$ are meromorphic functions on ${\Bbb C}$ such that each  $u_i|_{\mathcal Y}, 1\leq i\leq n,$ is a unit (i.e.  without  zeros and poles on ${\mathcal Y}$).  
	Let $F,G $ be coprime  polynomials in $n$ variables of the same degree $d>0$  over $K$.  Assume that at least  one of the coefficients in  the expansions of $F$ and $G$ is equal to 1.    
	Let $m\ge d$   and $b$ be a positive  integer.  If the set $\{u_1^{i_1}\cdots u_n^{i_n}: i_1+\cdots+i_n\le m\}$ is linearly  independent over $V_{F,G}(Mb+1)$,  then we have 
 \begin{align*}
		 N_{\rm gcd, \sigma}(F({\bf u}),G({\bf u}),r) 
		&\le_{\rm exc}  
		\left(\frac{M'}M+\frac{w}{u}-1\right)mn \max_{1\le i\le n} T^M_{u_i, \sigma}(r)+ O( \max_{1\le i\le n}\log T_{u_i}(r)) \\
		&+   c' (T_{ F, \sigma}(r)+T_{G, \sigma}(r)) +   c'' (\mathfrak{X}^+_{\sigma}(r)+ \log r +c_{1/h, \sigma}(r) + c_{1/W, \sigma}(r)),
	\end{align*}
	where $h$ is meromorphic function such that $F({\bf u})/h$ and $G({\bf u})/h$ are entire and have no common zero, and $W$ is  the Wronskian which arises when applying  Theorem \ref{movingsmt0},   $w=d_{Mb} $, $u=d_{M(b-1)}$,  $c' =\frac{w}{u}M(b+3)
	+2,$ and  $c''  =M{w^2\over u}$, and $M'= \binom{m+n}{n}-M$.
		\end{theorem}

		We outline a proof here.
\begin{proof} 
 For a subset
$T\subset K[\mathbf{x}]$, where $\mathbf{x}:=(x_{1},\ldots,x_{n})$, we denote
\[
T_{m}=\{f\in T\,:\deg f\le m\}.
\]
As in the proof of  \cite[Theorem  4.3]{GW22}, we assume that the ideal $(F, G)$ is proper in $K [x]$. Let $(F, G)_m := K [x]_m \cap (F, G)$, where
we choose $\{\phi_1,\dots, \phi_M\}$ to be a basis of the $K-$vector space $(F,G)_m$ consisting of elements of the form $F{\bf x}^{\bf i}, G{\bf x}^{\bf i}$. 
Let $\Phi:=(\phi_1,\dots, \phi_M)$. For each 
$z\in {\mathcal Y}$, we can construct a basis $B_z$ of $V_m := K[{\bf x}]_m/(F,G)_m$  with monomial representatives ${\bf x}^{{\bf i}_1}, \dots, {\bf x}^{{\bf i}_{M'}}$
as in the proof of Theorem 5.7 in \cite{levin2019greatest}. Let $I_z:=\{{\bf i}_1, \dots, {\bf i}_{M'}\}$. For each   ${\bf i}\not\in I_z$ with $|{\bf i}|=m$,  we have 
$${\bf x}^{\bf i}+\sum_{{\bf j}\in I_z} c_{z,  {\bf j}} {\bf x}^{\bf j} \in (F, G)_m$$
for some choice $c_{z,  {\bf j}}\in K$. 
Then, for each such ${\bf i}$,  there is a linear form   $\tilde L_{z, {\bf i}}$  over $K$ 
such that 
 $$\tilde L_{z, {\bf i}}(\phi_1,\dots, \phi_M) = {\bf x}^{\bf i}+\sum_{{\bf j}\in I_z} c_{z,  {\bf j}} {\bf x}^{\bf j}.$$
We note that since there are correspondingly only finitely many choices of a monomial basis of $V_m$, there are
only finitely many choices of $c_{z,  {\bf j}}$, even as $z$ ranges over all of   ${\mathcal Y}$.  Following the explicit construction in the proof of  \cite[Theorem 4.3]{GW22},  we can find a nonzero element $c_z\in  V_{F,G}(M)$ such that the normalized linear form $L_{z, {\bf i}}:=c_z \tilde L_{z, {\bf i}}$ satisfies
$$
L_{z, {\bf i}}(\phi_1,\dots, \phi_M) = c_z\left({\bf x}^{\bf i}+\sum_{{\bf j}\in I_z} c_{z,  {\bf j}} {\bf x}^{\bf j}\right) 
$$
and
\begin{align}\label{Linearformheight}
T_{L_{z, {\bf i}}, \sigma}(r) \leq M (T_{ F, \sigma}(r)+T_{G, \sigma}(r)),
\end{align}
by   \cite[(4.12)]{GW22}.
Moreover, the explicit construction also demonstrates that there are only a finitely many choice for $c_z$, even  though $z$ runs through all of ${\mathcal Y}$.
 
 Following the steps in the   proof of  Theorem 4.3 in \cite{GW22}, we then obtain 
\begin{align}\label{weilfunction}
\log |L_{z, {\bf i}}(\Phi({\bf u}))|-\log \|L_{z, {\bf i}}\|_z \leq \log|{\bf u}(z)^{\bf i}|+ \log \|F\|_z+\log\|G\|_z+O(1).
\end{align}
 Since the functions $u_i$  are units on $\mathcal Y$, and that
\begin{align}\label{logunit} 
\log^+|{\bf u}(z)^{\bf i}|\le \sum_{j=1}^n i_j \log^+|u_j(z)|\le m\sum_{j=1}^n  \log^+|u_j(z)|,	
\end{align}
we have 
\begin{align}\label{charfuncbound}
T^M_{{\bf u}^{\bf i}, \sigma}(r)\le m\sum_{j=1}^n T^M_{u_j, \sigma}(r)\le  mn\max_{1\le i\le n} T^M_{u_i, \sigma}(r).
\end{align}
Consequently, by \eqref{weilfunction}, we obtain
\begin{equation}\label{LW1}
 \int_{S^{\sigma}_r} \sum_{ |{\bf i}| = m, {\bf i}\not\in I_z} (\log |L_{z, {\bf i}}(\Phi({\bf u})| - \log \|L_{z, {\bf i}}\|_z)
 d\mu_r \leq M'mn\max_{1\le i\le n} T^M_{u_i, \sigma}(r)+ M(T_{ F, \sigma}(r)+T_{G, \sigma}(r)) +O(1).
\end{equation}
 Note that the map $\phi_1({\bf u}), \dots, \phi_M({\bf u})$  may not be a reduced presentation
of $\Phi({\bf u})$.   Let $h$ be a meromorphic function such that $F({\bf u})/h$ and $G({\bf u})/h$ are entire and have no common zero.  Let $\psi_i:=\phi_i({\bf u})/h$ for $i=1, \dots, M$. 
Then $\Psi: =(\psi_1, \dots, \psi_M)$  is a reduced representation of $\Phi({\bf u})$. Then
\begin{align}\label{compareht}
 \int_{S^{\sigma}_r} \log \|\Phi({\bf u})\| d\mu_r =  \int_{S^{\sigma}_r} \log \|\Psi\|  d\mu_r + \int_{S^{\sigma}_r} \log |h|d\mu_r  
 = T_{\Psi, \sigma}(r) + \int_{S^{\sigma}_r} \log |h|d\mu_r.
\end{align}
On the other hand, since $u_i$ are units on $\mathcal Y$, the poles of $h$ only arises from the  poles of coefficients of $F$ and $G$.  More precisely,
\begin{align}\label{pole_h}
N_{h, \sigma}(\infty,r)\le   T_{ F, \sigma}(r)+T_{G, \sigma}(r) 
\end{align}
by applying the inequality of the first main theorem in Theorem \ref{FMT}.
It is also clear that this choice of $h$ gives
\begin{align}\label{zero_h}
N_{h, \sigma}(0,r)=N_{\rm gcd, \sigma}(F({\bf u}),G({\bf u}),r).
\end{align}  
By Lemma \ref{jensenh}, \eqref{pole_h} and \eqref{zero_h},  we have
\begin{align}\label{h} 
 \int_{S^{\sigma}_r} \log |h|d\mu_r
  &= N_{h, \sigma}(0,r)-N_{h, \sigma}(\infty,r)+ c_{h,\sigma}(r)\cr
   &\ge N_{\rm gcd, \sigma}(F({\bf u}),G({\bf u}),r)- \left(T_{ F, \sigma}(r)+T_{G, \sigma}(r)\right)+c_{h, \sigma}(r). 
 \end{align}
Then we can derive from  \eqref{compareht} that
\begin{equation}\label{LW2} 
 \int_{S^{\sigma}_r} \log \|\Phi({\bf u})\|d\mu_r \ge T_{\Psi, \sigma}(r) + N_{\rm gcd, \sigma}(F({\bf u}),G({\bf u}),r) +c_{h, \sigma}(r)- \left(T_{ F, \sigma}(r)+T_{G, \sigma}(r)\right).
\end{equation} 
 Since the set $\{u_1^{i_1}\cdots u_n^{i_n}: i_1+\cdots+i_n\le m\}$ is linearly independent  over $V_{F,G}(Mb+1)$,
 it follows that $\Phi({\bf u}): {\Bbb C}\rightarrow {\Bbb P}^{M-1}$ 
 is linearly   independent  over $V_{F,G}(Mb) = V (b)$.  We  note that, when computing Weil functions, we  have the freedom to  assume that the coefficients of $L_{z,{\bf i}}$ are entire functions without common zeros, by multiplying a suitable meromorphic function.  By this construction, for each $z\in {\mathcal Y}$, the hyperplanes defined by evaluating the linear forms $L_{z,{\bf i}}$ at $z$ with 
  $|{\bf i}| = m$ and ${\bf i}\not\in I_z$ are in general position.    
For the computation  \eqref{Linearformheight} and  \eqref{LW1}, however, we will use the original $L_{z,{\bf i}}$. 
By applying  Theorem \ref{movingsmt0} to $\Phi({\bf u})$ and the set of the linear forms 
  $\{L_{z,{\bf i}}: |{\bf i}| = m, {\bf i}\not\in I_z\}$, we have the following inequality:
  \begin{align}\label{MSMT} &\int_{S^{\sigma}_r} \sum_{ |{\bf i}| = m, {\bf i}\not\in I_z} \lambda_{L_{z, {\bf i}}} (\Phi({\bf u}))d\mu_r +\frac 1uN_{W, \sigma}(0,r)\cr
		&\leq_{\operatorname{exc}}   \left(\frac wu M+\epsilon\right)T_{\Psi, \sigma}(r)+\frac wu M^2(b+2)  (T_{ F, \sigma}(r)+T_{G, \sigma}(r))  \cr
		 &+   M^2{w^2\over u}(\mathfrak{X}^+_{\sigma}(r)+ \log r)+ {1\over u}c_{1/W, \sigma}(r)+ O(\log \max_{1\le j\le n} T_{u_j}(r)), \end{align}
		 where $$\lambda_{L_{z, {\bf i}}} (\Phi({\bf u}))=-\log |L_{z, {\bf i}}(\Phi({\bf u}))|+ \log \|\Phi({\bf u})\|+ \log \|L_{z, {\bf i}}\|+ O(1),$$
		  and $W$  is the wronskian of $\psi_1,\dots, \psi_M$. 
		By combining \eqref{MSMT} with  \eqref{LW1}, \eqref{LW2} and use (\ref{h}),  we get 
		 \begin{align*} 
		& M N_{\rm gcd, \sigma}(F({\bf u}),G({\bf u}),r)  +\frac 1uN_{W, \sigma}(0,r)\cr &\leq_{\operatorname{exc}}  \left(\frac wu M-M+\epsilon\right)T_{\Psi, \sigma}(r)+ M'mn \max_{1\le i\le n} T^M_{u_i, \sigma}(r)
		\cr &+ 
		M\left(\frac wu M(b+2)+1\right) (T_{ F, \sigma}(r)+T_{G, \sigma}(r))\cr  
		& + M^2{w^2\over u} (\mathfrak{X}^+_{\sigma}(r)+ \log r)+{1\over u}c_{{1\over W}, \sigma}(r) + Mc_{{1\over h}, \sigma}(r) + O(1).\end{align*}  
Finally, we need to estimate $T_{\Psi, \sigma}(r)$ to complete the proof.
Recall that $\psi_i=\phi(\mathbf u)/h$ and $\phi_i$ is of the form $F{\bf x}^{\bf i}$, or $G{\bf x}^{\bf i}$, which is  a combination of monomial of degree $\le m$ with coefficients coming from $F$ or $G$.
Therefore, by \eqref{logunit} 
\begin{align*}
\log \|\psi_i\|\le m  \sum_{j=1}^n  \log^+|u_j|+\log\|F\|+\log\|G\|+\log \frac1h.
\end{align*}
Combining this with  Lemma \ref{jensenh} and \eqref{pole_h}, we obtain 
 $$
T_{\Psi, \sigma}(r)\le mn  \max_{1\le i\le n} T^M_{u_i, \sigma}(r) +2\left(T_{F,\sigma}(r)+T_{G,\sigma}(r)\right)+
   c_{\frac1h,\sigma}(r).
$$
\end{proof}
\begin{theorem} [The key theorem]\label{movinggcdn}  Let ${\mathcal Y}\subset {\Bbb C}$ be given  as above   with   $\sigma$  being associated to the discs $\{{\mathbb D}(a_j,  2r_j )\}$. 
Let $u_1, \dots, u_n$, with $n\ge 2$,  be  meromorphic functions on ${\Bbb C}$ such that they are  units on ${\mathcal Y}$.  Let $K$ be a subfield of the field of meromorphic functions on ${\Bbb C}$.  Let $F$ and $G$ be coprime  polynomials in $K[x_1,\hdots,x_n]$.   Assume that at least  one of the coefficients in  the expansions of $F$ and $G$ is equal to 1.    
Let $\epsilon>0$.
 Then, by shrinking $r_j$ small enough (depending on $\epsilon$, $u_i, 1\leq i\leq n$, and on $F$ and $G$), we can  choose   a positive integer  $m$ and positive reals $c_i$, $0\le i\le 5$, all depending only on $\epsilon$,  $n$, and the degree of $F$ and $G$ such that,
 either 
\begin{align}\label{gdegeneraten}
		T^{M}_{u_1^{m_1}\cdots u_n^{m_n}, \sigma}(r)  \le_{\rm exc}  c_1 \max_{1\le i\le n}\{\log T_{u_i}(r)\}+ c_2(T_{ F, \sigma}(r)+T_{G, \sigma}(r)) +c_3(\mathfrak{X}^+_{\sigma}(r)+ \log r)	\end{align}
	for some non-trivial tuple of integers $(m_1,\hdots,m_n)$ with $|m_1|+\cdots+|m_n|\le 2m$,  or
\begin{eqnarray}\label{gcd0n}
		N_{\rm gcd, \sigma} (F(u_1,\hdots,u_n),G(u_1,\hdots,u_n),r) &\le_{\rm exc} 
		\epsilon \max_{1\le i\le n} \{T_{u_i}(r)\} + c_4(T_{ F, \sigma}(r)+T_{G, \sigma}(r))\nonumber\\
		&+c_5(\mathfrak{X}^+_{\sigma}(r)+ \log r).
	\end{eqnarray}   
			  Furthermore, if $\epsilon>0$ is sufficiently small, then we may choose $m=C_1\epsilon^{-1}$, where $C_1$ is a positive real  independent of $\epsilon$. 
\end{theorem}
\begin{proof} We note that $c_i, 1\leq i\leq 5,$ will be determined during the course of the proof.
Recall that $M:=M_{m}:=2\binom{m+n-d}{n}-\binom{m+n-2d}{n}$  and $M':=M'_{m}:=\binom{m+n}{n}-M=O(m^{n-2})$.
   Based on the above estimates, we can  choose   a  real $C_1\ge 1$  independent of $\epsilon$ and the $u_i, 1\leq i\leq n, $ such that $  m= C_1\epsilon^{-1}\ge 2d$, and
	\begin{align}\label{findm} 
		\frac{M'mn}{M}\le\frac{\epsilon}{4}.
	\end{align}
	By \eqref{vlimit} we
	may then choose a sufficiently large integer $b\in\mathbb{N}$ such
	that 
	\begin{align}\label{wu}
		\frac{w}{u}-1\le\frac{\epsilon}{4mn},
	\end{align}
	where $w:=\dim_{{\bf k}}V_{F,G}(Mb)$ and $u:=\dim_{{\bf k}}V_{F,G}(Mb-M)$.

 Suppose that  the set $\{u_1^{i_1}\cdots u_n^{i_n}: i_1+\cdots+i_n\le m\}$ is linearly independent over $V_{F,G}(Mb+1)$.
	Then,  Theorem \ref{Mfundamental} implies that, by noticing \eqref{findm} and \eqref{wu}, 
 \begin{align*} 
		 N_{\rm gcd, \sigma}(F({\bf u}),G({\bf u}),r) 
		& \le_{\rm exc}\epsilon    \max_{1\le i\le n} \{T^M_{u_i, \sigma}(r)\}+ {\tilde c} \max_{1\le i\le n}\{\log T_{u_i}(r)\}\\
		&~+ c_4(T_{ F, \sigma}(r)+T_{G, \sigma}(r))+c_5  \left(\mathfrak{X}^+_{\sigma}(r)+ \log r+  c_{{1\over h}, \sigma}(r) + c_{{1\over W}, \sigma}(r)\right), 
	\end{align*}
 where  ${\tilde c}$ and $c_i$, $4\le i\le 5$, are computable positive reals.
 
 If  the set $\{u_1^{i_1}\cdots u_n^{i_n}: i_1+\cdots+i_n\le m\}$ is linearly   dependent over $V_{F,G}(Mb+1)$,
i.e. there is a non-trivial
relation 
\begin{align}\label{uniteq}
\sum_{\mathbf{i}\in I}\alpha_{\mathbf{i}}\mathbf{u}^{\mathbf{i}}=0, 
\end{align}
where the sum runs over those $\mathbf{i}\in\mathbb{Z}_{\ge0}^{n}$
and $|\mathbf{i}|\le m$, $\alpha_{\mathbf{i}}\in V_{F,G}(Mb+1)$
for each $\mathbf{i}$, and $\alpha_{\mathbf{i}}\ne 0$ for $\mathbf{i}\in I$.  Without loss of generality, we assume that no proper subsum of \eqref{uniteq} vanishes.
 Then Theorem  \ref{trunborel}  implies that
\begin{align}\label{truncountn}
 T^{M}_{\frac{\alpha_{\mathbf{i}}}{\alpha_{\mathbf{i}_{0}}}\mathbf{u}^{\mathbf{i}-\mathbf{i}_{0}}, \sigma}(r)
&\leq_{\exc}  \sum_{\mathbf{i}\in I}  N_{\beta\alpha_{\mathbf{i}},\sigma}(0,r)+c_3 (\mathfrak{X}^+_{\sigma}(r)+  \log r)\nonumber \\
& +c_{\frac1{D_0},\sigma}(r)+ \sum_{\substack{\mathbf{i}\in I \\ \mathbf{i}\neq \mathbf{i}_0}}  
c_{\beta\alpha_{\mathbf{i}}{\bf u}^{\bf i}, \sigma}(r)+ O(\max_{0\leq i\leq n}\{\log T_{u_i} (r)\})+O(\log (T_F(r)+\log F_G(r)), 
\end{align}
where $c_3={\binom{n+m}{n}}\left({\binom{n+m}{n}}+1\right)$. 
We note that 
\begin{align}\label{truncount}
 N_{\beta\alpha_{\mathbf{i}},\sigma}(0,r)\le N_{ \alpha_{\mathbf{i}},\sigma} (0,r)+N_{\beta,\sigma} (0,r)\le 2(Mb+1)(T_{ F, \sigma}(r) +T_{G, \sigma}(r))+c_{{1\over \alpha_{\mathbf{i}}}, \sigma}(r),
\end{align}
since the  zero of $\beta$ comes from the pole of $\alpha_{\mathbf{i}}$ and by \eqref{rci}.  Also,   since $\alpha_{\mathbf{i}}\in V_{F,G}(Mb+1)$, 
 \begin{align}\label{htt}
T^{M}_{\frac{\alpha_{\mathbf{i}{0}}}{\alpha_{\mathbf{i}}} , \sigma}(r)  &\le T^M_{1/\alpha_{\mathbf{i}} , \sigma}(r)+T^M_{\alpha_{\mathbf{i_0}}, \sigma}(r)
\le T^M_{ \alpha_{\mathbf{i}} , \sigma}(r)+T^M_{ \alpha_{\mathbf{i_0}} , \sigma}(r) + c_{{1\over  \alpha_{\mathbf{i}}}, \sigma}(r) \nonumber \\
& \le 2(Mb+1) (T_{ F, \sigma}(r)+T_{G, \sigma}(r))+c_{{1\over  \alpha_{\mathbf{i}}}, \sigma}(r).
\end{align}
 Then,  by combining (\ref{truncountn}),  (\ref{truncount}), (\ref{htt}),
\begin{align*}
T^{M}_{\mathbf{u}^{\mathbf{i}-\mathbf{i}_{0}}, \sigma}(r)  &\le T^{M}_{\frac{\alpha_{\mathbf{i}}}{\alpha_{\mathbf{i}_{0}}}\mathbf{u}^{\mathbf{i}-\mathbf{i}_{0}}, \sigma}(r)+ T^{M}_{\frac{\alpha_{\mathbf{i}_{0}}}{\alpha_{\mathbf{i}}} , \sigma}(r)  
\le_{\rm exc}   c_2 (T_{ F, \sigma}(r)+T_{G, \sigma}(r))+c_3 (\mathfrak{X}^+_{\sigma}(r) + \log r ) \cr
&+  c_{\frac1{D_0},\sigma}(r) + 2 c_{{1\over \alpha_{\mathbf{i}}}, \sigma}(r) + \sum_{\substack{\mathbf{i}\in I \\ \mathbf{i}\neq \mathbf{i}_0}}  
c_{\beta\alpha_{\mathbf{i}}{\bf u}_{\mathbf{i}}, \sigma}(r) +O(\max_{0\leq i\leq n}\{\log T_{u_i} (r)\}), 
\end{align*}
where  $c_2=4 (Mb+1)$. 

 By Lemma \ref{Ch}, we can shrink $r_j$ small enough so that all of the 
quantities
\[
c^{+}_{1/h,\sigma}(r), \quad 
c^{+}_{1/W,\sigma}(r), \quad
c_{1/D_0,\sigma}(r), \quad
c_{1/\alpha_{\mathbf{i}},\sigma}(r) \ \ ( \mathbf{i}\in I), \quad
\sum_{\substack{\mathbf{i}\in I \\ \mathbf{i}\ne \mathbf{i}_0}}
c_{\beta \alpha_{\mathbf{i}} \mathbf{u}^{\mathbf{i}},\,\sigma}(r)
\]
are bounded by \(O(1)\), where the implied constant depends on $\mathbf{u}$, $F$, 
and $G$, but is independent of $r$. 
	 In addition, we  shrink $r_j$ small enough such that (\ref{Amero}) and the Proposition \ref{exhuastion} for $u_i, 1\leq i\leq n$, hold, i.e. $T_{u_i}(r)$ and $T^M_{u_i,  \sigma}(r)$ differ only by  $\log r+O(1)$. This proves the  key theorem. 
	\end{proof}

For the convenience of later application, we state the following result for $n=1$, which can be  easily obtained from the theory of resultant (see \cite[Proposition 10]{RuWang2024}). 
 \begin{proposition}  \label{gcdn1}
  Let $K$ be a subfield of the field of meromorphic functions on ${\Bbb C}$.  Let $F,G\in K[x]$ be nonconstant coprime  polynomials.  Let ${\mathcal Y}\subset {\Bbb C}$ be given as above with  $\sigma$  being associated to the discs $\{{\mathbb D}(a_j,  2r_j )\}$.     Then there exists a positive constant $c$ such that for any  meromorphic function $g$ on ${\mathcal Y}$, the following holds:
	\begin{align*} 
		N_{\rm gcd,\sigma} (F(g),G(g),r)\le  c(T_{ F, \sigma}(r)+T_{G, \sigma}(r)).  
	\end{align*} 
\end{proposition}

\section{The abc theorem for parabolic Riemann surfaces}\label{abcparabolic}
 Let ${\mathcal Y}\subset {\Bbb C}$ be given as above with $\sigma$  being associated to the discs $\{{\mathbb D}(a_j,  2r_j )\}$.  Let ${\bf u}=(u_1,\hdots,u_n)$
where $u_1,\hdots,u_n$ are meromorphic functions on ${\Bbb C}$ such that each restriction $u_i|_{\mathcal Y}, 1\leq i\leq n,$ is a unit.
 Let $\mathbb K_{\mathbf u}$ be the set of meromorphic function  $a$   on ${\Bbb C}$ for which   there exists a nonnegative  constant $c_a$ (depending on $a$) with
\begin{align}\label{Ku}
T^M_{a, \sigma}(r) \le_{\rm exc}   c_a\left(\mathfrak{X}^+_{\sigma}(r)+ \log r +\max_{1\le i\le n}\{\log T_{u_i}(r)\}\right),
\end{align} 
 after shrinking radii $r_j$ suffciently small  (depending on $a$).  By   \eqref{rci} and Lemma \ref{Ch}, it is easy to see that  $\mathbb K_{\mathbf u}$
 is a  subfield of the meromorphic function field.
Notably, in our application, the constant $c_a$ can be effectively computed.    We will use Theorem  \ref{movinggcdn} for $K:=\mathbb K_{\mathbf u}$.
  Note that, by Lemma \ref{DL} and by \eqref{rci},  if $a\in \mathbb K_{\mathbf u} $, then $a'\in \mathbb K_{\mathbf u}$,
since
\begin{align*}
T^M_{a', \sigma}(r) &\le T^M_{\frac{a'}{a}, \sigma}(r)+ T^M_{a, \sigma}(r)\le m_{a'/a, \sigma}(r)+N_{a,  \sigma}(0, r)+N_{a,  \sigma}(\infty, r)+ T^M_{a, \sigma}(r)\\
&\le  m_{a'/a, \sigma}(r)+ T^M_{\frac 1a, \sigma}(r)+ 2T^M_{a, \sigma}(r).
\end{align*}

Let  $\mathbf{x}:=(x_{1},\ldots,x_{n})$.  For $\mathbf{i}=(i_{1},\ldots,i_{n})\in\mathbb{Z}^{n}$, we let $\mathbf{x^{i}}:=x_{1}^{i_{1}}\cdots x_{n}^{i_{n}}$ and  $\mathbf{u^{i}}:=u_{1}^{i_1}\cdots u_{n}^{i_{n}}$. For a non-constant polynomial $F(\mathbf{x})=\sum_{\mathbf{i}}a_{\mathbf{i}}\mathbf{x}^{\mathbf{i}}\in   \mathbb K_{\mathbf u} [\mathbf{x}]:=      \mathbb K_{\mathbf u}[x_1,\dots,x_n]$,
we define 
\begin{align}\label{Duexpression}
	D_{\mathbf{u}}(F)(\mathbf{x}):=\sum_{\mathbf{i}}\frac{(a_{\mathbf{i}}\mathbf{u}^{\mathbf{i}})'}{\mathbf{u}^{\mathbf{i}}}\mathbf{x}^{\mathbf{i}}
	=\sum_{\mathbf{i}}\left(a_{\mathbf{i}}'+a_{\mathbf{i}}  \cdot\sum_{j=1}^{n}i_j\frac{u_{j}'}{u_{j}}\right)\mathbf{x}^{\mathbf{i}}\in     \mathbb K_{\mathbf u}[\mathbf{x}].
\end{align} 
A direct computation shows that 
\begin{align}\label{fuvalue}
	F(\mathbf{u})'=D_{\mathbf{u}}(F)(\mathbf{u}),
\end{align}
and that  the following product rule 
\begin{align}\label{productrule}
	D_{\mathbf{u}}(FG)=D_{\mathbf{u}}(F)G+FD_{\mathbf{u}}(G)
\end{align}
holds for  $F,G\in \mathbb K_{\mathbf u}[\mathbf{x}]$.

 By Lemma  \ref{unit},  we see that $u_i'/u_i \in \mathbb K_{\mathbf u}$.  
In the following lemma and in Theorem~\ref{prethm}, we will identify certain 
monomial products of bounded degree in the $u_i$ that satisfy \eqref{Ku}, and 
hence are contained in $\mathbb K_{\mathbf u}$.

The following lemma shares the same proof of Lemma 5.1 in  \cite{GW22}.
\begin{lemma}\label{coprimeD} 
Let $\mathbf u=(u_1,\hdots,u_n)$, where $u_i$, $1\le i\le n$, are the units  on ${\mathcal Y}$.  
  Let $G$ be a non-constant     polynomial in $  \mathbb K_{\mathbf u}[\mathbf{x}]$  with no monomial factors and no repeated factors,  and assume that at least one of its coefficients is equal to 1.
		Then   $F:=D_{\mathbf{u}}(G)$ and $G$ 
	are coprime  in  $  \mathbb K_{\mathbf u}[\mathbf{x}]$ unless there exists a non-trivial tuple of integers $(m_1,\hdots,m_n)$ with $\sum_{i=1}^n|m_i|\le 2\deg G$ such that  
 $T^M_{ u_1^{m_1} \cdots u_n^{m_n}, \sigma}\le T_{G, \sigma}(r).$ 
\end{lemma}
\begin{proof}
 Relation \eqref{productrule} implies that we may assume, as  in the first step of the proof of  \cite[Lemma 5.1]{GW22},  that $G$ is irreducible. 
Write $G=\sum_{\mathbf{i}}b_{\mathbf{i}}\mathbf{x}^{\mathbf{i}}\in \mathbb K_{\mathbf u}[\mathbf{x}]$, which contains at least two distinct terms, since $G$ has no monomial factors. 
	Then   $D_{\mathbf{u}}(G) /G\in\mathbb C^*$  since  $D_{\mathbf{u}}(G)$ is not zero and $\deg D_{\mathbf{u}}(G)=\deg G$.
	Comparing the coefficients of $G$ and $D_{\mathbf{u}}(G)$, for nonzero $b_{\mathbf{i}}$ and  $b_{\mathbf{j}}=1$  in $ \mathbb K_{\mathbf u}$ ($\mathbf{i}\ne\mathbf{j}$), we have from \eqref{Duexpression} that
	\begin{equation*}\label{eq_diff1}
		\frac{(b_{\mathbf{i}}\mathbf{u}^{\mathbf{i}})'}{b_{\mathbf{i}}\mathbf{u}^{\mathbf{i}}}=\frac{ (\mathbf{u}^{\mathbf{j}})'} {\mathbf{u}^{\mathbf{j}}},
	\end{equation*} 
	where $\mathbf{i}=( i_1,\hdots,i_n)$, and $\mathbf{j}=( j_1,\hdots,j_n)$.	It implies that 
	$b_{\mathbf{i}}\frac {\mathbf{u}^{\mathbf{i}}}{ \mathbf{u}^{\mathbf{j}}}\in\mathbb C$.  Therefore,
	$$
T^M_{ u_1^{j_1-i_1} \cdots u_n^{j_n-i_n},\sigma}\le T_{G, \sigma}(r).
	$$
\end{proof}
As an application of the  key theorem (Theorem  \ref{movinggcdn})  and using the above lemma, we obtain the following theorem.
 \begin{theorem}\label{prethm}   Let ${\mathcal Y}\subset {\Bbb C}$ be given as above  with   $\sigma$ being associated to the  discs  $\{{\mathbb D}(a_j,  2r_j )\}$.
 Let ${\bf u}=(u_1, \dots, u_n)$, where each $u_i$, $1\le i\le n$, is a meromorphic function on ${\Bbb C}$  which is a unit on ${\mathcal Y}$. 
   Let $G$ be a non-constant  polynomial in $\mathbb K_{\mathbf u}[\mathbf{x}]$  with no monomial factors and no repeated factors.  Assume that at least  one of the coefficients in  the expansions of  $G$ is equal to 1.  
 Then,  by shrinking $r_j$ small enough (depending on $\epsilon$, ${\bf u}$ and $G$), we can find   a positive integer    $m$  and   positive constants  $\tilde c_1$ and $\tilde c_2$ depending only on $\epsilon$,  $n$, and on the degree and the coefficients of $G$, but  independent of $\mathbf u$,  for which the following holds:
 Either there exists a non-trivial $n$-tuple $(m_1,\hdots,m_n)$ of integers with $\sum_{i=0}^n |m_i|\le  2m$ such that 
 \begin{align}\label{(ii)}
T^{M}_{ u_1^{m_1}\cdots u_n^{m_n},\sigma}(r) \le_{\rm exc}  \tilde c_1(\mathfrak{X}^+_{\sigma}(r)+ \log r) +O(\max_{1\le j\le n}\{\log T_{u_j}(r)\}),
	\end{align}
or else
 \begin{align}\label{(new)}
N_{G(\mathbf{u}), \sigma}(0,r)-N^{(1)}_{G(\mathbf{u}), \sigma}(0,r)\le_{\rm exc} 
\epsilon \max_{1\le j\le n }\{T_{u_j}(r)\}+ \tilde c_2(\mathfrak{X}^+_{\sigma}(r)+ \log r).
\end{align}
  \end{theorem}
  
  \begin{proof}
 We first note that the constant ${\tilde c}_1, {\tilde c}_2$    will be determined during the course of the proof. 
Let $z_0\in\mathcal Y$.  If $v_{z_0}( G(\mathbf{u}))\ge 2$, then it follows from \eqref{fuvalue} that $v_{z_0}( D_{\mathbf{u}}(G)(\mathbf{u}))= v_{z_0}( G(\mathbf{u}))-1$.
Hence,
$$
\min\{v_{z_0}^+(G(\mathbf{u})),v_{z_0}^+(D_{\mathbf{u}}(G)(\mathbf{u}))\}\ge v_{z_0}^+( G(\mathbf{u}))-\min\{1,v_{z_0}^+( G(\mathbf{u}))\}.
$$
Consequently,
\begin{align}\label{truncate}
N_{\gcd, \sigma}(G(\mathbf{u}), D_{\mathbf{u}}(G)(\mathbf{u}),r)\ge N_{G(\mathbf{u}), \sigma}(0,r)-N^{(1)}_{G(\mathbf{u}), \sigma}(0,r).
\end{align}
Write $G= \sum_{{\bf i}} a_{{\bf i}}{\bf x}^{{\bf i}}$ with  $a_{{\bf i}}\in \mathbb K_{\mathbf u}$. Then, from the definition of $\mathbb K_{{\bf u}}$, 
 there exists a nonnegative real  $c_G$  (depending on $G$)  such that 
 \begin{align}\label{heightG}
T_{ G, \sigma}(r)\le_{\rm exc}   c_G\left(\mathfrak{X}^+_{\sigma}(r)+ \log r +\max_{1\le i\le n}\{\log T_{u_i}(r)\}\right) 
\end{align}  
 after shrinking radii $r_j$ suffciently small  (depending on $G$).  
Furthermore,  by  Lemma \ref{DL}, we have
\begin{align}\label{heightDG}
T_{D_{\mathbf{u}}(G), \sigma}(r)&\le   2T_{ G, \sigma}(r)+\deg G \sum_{i=1}^n T^M_{u_i'/u_i, \sigma}(r) \nonumber \\
&\le_{\rm exc} (2c_G + n\deg G )(\mathfrak{X}^+_{\sigma}(r)+\log r) + O(\max_{1\le i\le n}\{\log T_{u_i}(r)\}).
\end{align} 
 By Lemma \ref{coprimeD} and \eqref{heightG},  either  $G$ and $D_{\mathbf{u}}(G)$ are coprime or   or there exists a non-trivial $n$-tuple $(m_1,\ldots,m_n)\in\mathbb Z^n$ with $\sum_{i=0}^n |m_i|\le  2\deg G$
such that
 \begin{align*} 
T^M_{u_1^{m_1}\cdots u_n^{m_n},\sigma}(r) 
\le_{\rm exc} T_{G, \sigma}(r)\le_{\rm exc}   c_G(\mathfrak{X}^+_{\sigma}(r)+ \log r)+O(\max_{1\le i\le n}\{\log T_{u_i}(r)\}).
	\end{align*}

   It remains to consider where $G$ and $D_{\mathbf{u}}(G)$ are coprime.  
Let $\epsilon>0$ be given, which may assume  to be sufficiently small.
 By apply Theorem \ref{movinggcdn}, after  shrinking $r_j$ small enough (depending on $\epsilon$, ${\bf u}$ and $G$),  we find  $m$ and $ c_i>0$, $1\le i\le 4$,  such that  
either 
 \begin{align}\label{gdegenerate1}
		  T^M_{ u_1^{m_1}\cdots u_n^{m_n},\sigma}(r) \le_{\rm exc}   c_1(T_{ G, \sigma}(r)+ T_{D_{\mathbf{u}}(G), \sigma}(r)) +  c_2(\mathfrak{X}^+_{\sigma}(r)+ \log r) 
	\end{align}
 for some non-trivial tuple of integers $(m_1,\hdots,m_n)$ with $|m_1|+\cdots+|m_n|\le 2m$, 
or 
  \begin{align}\label{gcd3}
&N_{\rm gcd, \sigma} (G({\mathbf{u}}), D_{\mathbf{u}}(G)({\mathbf{u}}),r)  \nonumber \\
&\le_{\rm exc}{\epsilon\over 2}\max_{1\le i\le n}T_{u_i}(r) + c_3(T_{ G, \sigma}(r)+ T_{D_{\mathbf{u}}(G), \sigma}(r))+c_4(\mathfrak{X}^+_{\sigma}(r)+ \log r).
\end{align} 
We can conclude the proof by combining \eqref{heightG},   \eqref{heightDG}, \eqref{gdegenerate1}, and \eqref{gcd3}. 
\end{proof}

The proof of \cite[Theorem 4]{GNSW} can be adapted to the current situation, leading to the following result.  
 
 \begin{theorem}\label{mainthmK}
Let $G$ be a non-constant  polynomial in $\mathbb C[x_1,\hdots,x_n]$ with no monomial  factors and no repeated factors. 
Then, for any $\epsilon>0$,   there exist a positive real number $c_0$, a proper Zariski closed subset $Z$ of $\mathbb{A}^n(\mathbb C)$  expressible as the zero locus of a finite set
$\Sigma\subset \mathbb C[x_1,\ldots,x_n]$ with the following property: 
Let ${\bf u}=(u_1, \dots, u_n)$, where  each $u_i$, $1\le i\le n$, is meromorphic function on ${\Bbb C}$, with $H({\bfu})\ne 0$ for all $H\in\Sigma$ and $G(\mathbf u)\ne 0$. Let ${\mathcal Y}\subset {\Bbb C}$ be given as above  with   $\sigma$  being associated to the discs $\{{\mathbb D}(a_j,  2r_j )\}$.  Assume that $u_1, \dots, u_n$ are units on ${\mathcal Y}$.
Then, by shrinking $r_j$ small enough (depending on ${\bf u}$ and $G$), we have
\begin{align}\label{GCD1}  
N_{G(\mathbf u),\sigma}(0,r)-N^{(1)}_{G(\mathbf u),\sigma}(0,r)) \le_{\rm exc} \epsilon\max_{1\le i\le n}\{T_{u_i}(r)\} +c_0(\mathfrak{X}^+_{\sigma}(r)+ \log r).
\end{align}   
 Here, $c_0$ can be  effectively bounded from above in terms of $\epsilon$, $n$, and the degree of $G$.    Moreover,  the finite set
$\Sigma\subset \mathbb C[x_1,\ldots,x_n]$ satisfies: {\rm(Z1)} $\Sigma$ depends on $\epsilon$ and $G$, but not on ${\bf u}$ or $\mathcal Y$, and can be determined explicitly, {\rm(Z2)} $\vert \Sigma\vert$ and the degree of each polynomial in $\Sigma$ can be effectively bounded from above in terms of $\epsilon$, $n$, and the degree of $G$.  
 \end{theorem}
 We outline a proof here.
 \begin{proof} 
  We follow the proof of \cite[Theorem 4]{GNSW}.   We apply Theorem \ref{prethm} through several reduction steps to derive a contradiction. The constant $c_0$ will be determined in the course of the proof. 
We also remark that, since the coefficients of $G$ lie in $\mathbb{C}$, the assumption that at least one coefficient in the expansion of $G$ is equal to $1$ imposes no restriction.
By Theorem \ref{prethm}, after shrinking $r_j$, 
  we have either 
 \begin{equation}\label{other}T^{M}_{ u_1^{m_1}\cdots u_n^{m_n},\sigma}(r) \le_{\rm exc}  \tilde c_1(\mathfrak{X}^+_{\sigma}(r)+ \log r) +O(\max_{1\le j\le n}\{\log T_{u_j}(r)\})\end{equation}
or
 \begin{align}\label{(newn)}
N_{G(\mathbf{u}), \sigma}(0,r)-N^{(1)}_{G(\mathbf{u}), \sigma}(0,r)\le_{\rm exc} 
\epsilon \max_{1\le j\le n }\{T_{u_j}(r)\}+ \tilde c_2(\mathfrak{X}^+_{\sigma}(r)+ \log r).
\end{align}
 Let $c_0 \ge \tilde c_2$.   If (\ref{(newn)}) holds, then we are done.   Assume it is  not, then (\ref{other}) holds for a non-trivial $n$-tuple $(m_1,\hdots,m_n)$ of integers with $\sum\vert m_i\vert\leq M_1$, where 
   $M_1$,  depending only on $\epsilon$,  $n$, and the degree of $G$, but  independent of $\mathbf u$. Thus, we have
   \begin{equation}\label{Ru2}
  \lambda_1:={u_1}^{m_1}\cdots  {u_n}^{m_n} \in \mathbb K_{\mathbf u}.\
  \end{equation}   
      We may assume $\gcd(m_1,\ldots,m_n)=1$. By Lemma 14 in \cite{GNSW}, $(m_1,\hdots,m_n)$ extends to a basis 
$(m_1,\hdots,m_n)$,  $(a_{21},\hdots,a_{2n}),\hdots, (a_{n1},\hdots,a_{nn})$
of  $\mathbb Z^n$ with
\begin{align}\label{basisbound_step1}
|a_{i1}|+\cdots+|a_{in}|\le M_1+n\quad\text{ for } 2\le i\le n.
\end{align} 
Consider the change of variables
\begin{align}\label{transform_var_step1}
\Lambda_1:=x_1^{m_1}\cdots x_n^{m_n}, \quad\text{and } \quad
X_{1,i}:=x_1^{a_{i1}}\cdots x_n^{a_{in}}  \quad\text{for } 2\le i\le n
\end{align}
and put
\begin{align}\label{transform_unit_step1}
\beta_{1,i}=u_1^{a_{i1}}\cdots u_n^{a_{in}}  \quad\text{for } 2\le i\le n.
\end{align}
Let $A_1$ denote the $n\times n$ matrix whose rows are the above basis of $\mathbb{Z}^n$. Then we formally express the above identities as
\begin{equation}\label{eq:formal_A1}
(\Lambda_1,X_{1,2},\ldots,X_{1,n})=(x_1,\ldots,x_n)^{A_1}\quad \text{and} \quad (\lambda_1,\beta_{1,2},\ldots,\beta_{1,n})=(u_1,\ldots,u_n)^{A_1}.
\end{equation}
Let $B_1=A_1^{-1}$. The entries of $B_1$ can be bounded from above in terms of $M_1$ and $n$. We have
\begin{equation}\label{eq:formal_B1}
(x_1,\ldots,x_n)=(\Lambda_1,X_{1,2},\ldots,X_{1,n})^{B_1}\quad \text{and} \quad (u_1,\ldots,u_n)=(\lambda_1,\beta_{1,2},\ldots,\beta_{1,n})^{B_1}.
\end{equation}
Let $G_1(\Lambda_1,X_{1,2},\ldots,X_{1,n})\in {\Bbb C}[\Lambda_1,X_{1,2},\ldots,X_{1,n}]$ with no monomial factors
and 
\begin{equation}\label{eq:G_1}
G((\Lambda_1,X_{1,2},\ldots,X_{1,n})^{B_1})=\Lambda_1^{d_1}X_{1,2}^{d_2}\cdots X_{1,n}^{d_n} G_1(\Lambda_1,X_{1,2},\ldots,X_{1,n})
\end{equation}
for some integers $d_i$, $1\leq i\leq n$. Since the transformations in \eqref{eq:formal_A1} and \eqref{eq:formal_B1} are inverses of each other, and $G$ has no repeated irreducible factors, it follows that $G_1$ also has no repeated irreducible factors. The coefficients of $G_1$ are the same as the coefficients of $G$ and $\deg(G_1)$ can be bounded from above explicitly in terms of $M_1$, $n$, and $\deg(G)$. Consider $G_1(\lambda_1,X_{1,2},\ldots,X_{1,n})\in  \mathbb C(\lambda_1)[X_{1,2},\ldots,X_{1,n}]\subset {\mathbb K_{\mathbf u}}[X_{1,2},\ldots,X_{1,n}]$. 

  For the particular change of variables in \eqref{eq:formal_A1}, \eqref{eq:formal_B1}, and \eqref{eq:G_1} (that depends on the matrix $A_1$),
we apply the Lemma 16 in \cite{GNSW}  to find a nonconstant polynomial $\widetilde H_1\in \mathbb C[x_1,\hdots,x_n]$ 
such that   $G_1(\lambda_1,X_{1,2},\ldots,X_{1,n})$
has neither monomial nor repeated irreducible factors  if $\widetilde H_1(u_1,\ldots,u_n)\not\equiv  0$. 
It is important to note that the polynomial $\widetilde H_1$ is not only independent of $\mathbf u$, but also independent of $\mathcal Y$.
We now take $H_1$ to be the product of all
such $\widetilde H_1$ where $A_1$ ranges over the finitely many elements of $\GL_n(\mathbb{Z})$ with $\Vert A_1\Vert_{\infty}\leq M_1+n$. From Lemma 16 in \cite{GNSW}\, we know that $\deg H_1$ depends only on $\epsilon$, $n$ and $\deg G$. 

  Since the $u_i$'s, $\lambda_1$, and $\beta_{1,j}$'s are the units on $\mathcal Y$, we have
\begin{align}\label{eq:N_S-N_S^1 step1}
\begin{split}
  N_{G(\mathbf u),\sigma}(0,r)-N_{G(\mathbf u),\sigma}^{(1)}(0,r) 
=  N_{G_1(\lambda_1,\beta_{1,2},\ldots,\beta_{1,n}),\sigma}(0,r)-N_{G_1(\lambda_1,\beta_{1,2},\ldots,\beta_{1,n}),\sigma}^{(1)}(0,r) .
\end{split}
\end{align}
By  \eqref{eq:formal_A1} and  \eqref{eq:formal_B1},  we have:
\begin{equation}\label{htreduction}
 \max_{1\leq i\leq n}\{T_{u_i}(r)\} ={\rm O}\left(\max\{T_{\lambda_1}(r), T_{\beta_{1,2}}(r),  \ldots,T_{\beta_{1,n}}(r) \} \right).
\end{equation} 
 By \eqref{basisbound_step1}, there are only a finitely many choice of  $\beta_{1,t}$, $2\le t\le n$.  Hence, by Proposition~\ref{exhuastion}, after possibly shrinking  $r_j$ (depending on $u_i$ ), we may assume that
 \[
\bigl|T_{u_i}(r)-T^M_{u_i,\sigma}(r)\bigr|,\quad
\bigl|T_{\beta_{1,t}}(r)-T^M_{\beta_{1,t},\sigma}(r)\bigr|,\quad\text{and}\quad
\bigl|T_{\lambda_1}(r)-T^M_{\lambda_1,\sigma}(r)\bigr|
\]
are all bounded by $\log r+O(1)$ for $1\le i\le n$ and  $2\le t\le n$.
Then by  \eqref{htreduction} and \eqref{other}, we find positive real numbers $b_{1,1}$ and $b_{1,2}$ such that
\begin{align*} 
& b_{1,1} \max_{1\leq i\leq n}\{T_{u_i}(r)\}  \le_{\rm exc} \max_{2\leq t\leq n}\{T_{\beta_{1,t}}(r)\}
 + \tilde c_1(\mathfrak{X}^+_{\sigma}(r)+ \log r),\quad\text{and} \\ 
  &\max_{2\leq t\leq n}\{T_{\beta_{1,t}}(r)\} \le b_{1,2} \left( \max_{1\leq i\leq n}\{T_{u_i}(r)\} +\log r \right).
\end{align*}  
In conclusion,  at the end of this step we have that, after possibly shrinking  $r_j$, 
 the inequality
\begin{align}\label{eq:end step1 2}
 N_{G_1(\lambda_1,\beta_{1,2},\ldots,\beta_{1,n}),   \sigma }(0,r)-N_{G_1(\lambda_1,\beta_{1,2},\ldots,\beta_{1,n}),  \sigma }^{(1)}(0,r)&\le_{\rm exc} \frac{\epsilon}{ b_{1,2}} \max_{2\le i\le n}  \{T_{\beta_{1, i}}(r)\} 
 + c_0(\mathfrak{X}^+_{\sigma}(r)+ \log r)
\end{align}
fails to hold under the assumption that $H_1(u_1,\ldots,u_n)\not\equiv 0$. 

  The above is the first step. The next step is to apply Theorem \ref{prethm} to $G_1$ and ${\bf u}$.  Note that, to do so, we need to  further shrink  $r_j$ according to $G_1$. There are $(n-1)$ many steps in total as in the proof of \cite[Theorem 4]{GNSW}.
Let $2\leq s\leq n-1$ and suppose that we have completed the step $s-1$. This includes the construction
of $H_{s-1}\in {\mathbb C}[x_1,\hdots,x_n]$ with degree depends on $\epsilon$, $n$ and $\deg G$ only.   The last one is Step $(n-1)$ resulting in $H_{n-1}\in {\mathbb C}[x_1,\hdots,x_n]$. We now define
$H=H_1 \cdots H_{n-1}$.  Then $\deg H$  depends only on $\epsilon$, $n$ and $\deg G$, since each $H_i$ does so.  Moreover, the polynomial $H$ is not only independent of $\mathbf u$, but also independent of $\mathcal Y$.

 Suppose $H(u_1,\ldots,u_n)\not\equiv 0$.   
Assume we go through all the $(n-1)$ steps
to get the polynomial
$$P(X_{n-1,n}):=G_{n-1}(\lambda_1,\ldots,\lambda_{n-1},X_{n-1,n})\in {\mathbb K_{\mathbf u}}[X_{n-1,n}]$$ 
such that its degree can be bounded explicitly in terms of $M_{n-1},\ldots,M_1$, $n$, and $\deg(G)$.  At the end of the step $(n-1)$, we have that $\beta:=\beta_{n-1,n}$,  which is a (non-constant) unit on ${\mathcal Y}$,  such that 
the inequality,  after further shrinking $r_j$ according to $G, G_1, \dots, G_{n-1}$,
\begin{align}\label{Ru3}
N_{P(\beta ), \sigma}(0,r)-N^{(1)}_{P(\beta ), \sigma}(0,r)
\le_{\rm exc} \frac{\epsilon}{b_{n-1,2}}  T_{\beta}(r) + c_0(\mathfrak{X}^+_{\sigma}(r)+ \log r)
\end{align}
fails to hold.
Here, $b_{n-1,1}$ and $b_{n-1,2}$ are computable positive real numbers.
 On the other hand, since $H_{n-1}(u_1,\ldots,u_n)\not\equiv 0$, the polynomial $P(X_{n-1,n})$
has neither monomial nor repeated irreducible factors, and thus,  similar to \eqref{truncate}, we have
\begin{align}\label{truncate0}
N_{\gcd, \sigma}(P(\beta ), D_{\beta}(P)(\beta ),r)\ge N_{P(\beta ), \sigma}(0,r)-N^{(1)}_{P(\beta ),  \sigma}(0,r).
\end{align}
By Proposition \ref{gcdn1}, there is a positive constant $C_1$ independent of $\beta$ such  that
\begin{align}\label{lastgcd}
N_{\gcd, \sigma}(P(\beta ), D_{\beta}(P)(\beta ),r)\le  C_1(T_{ P, \sigma}(r)+T_{D_{\beta}(P), \sigma}(r)).  
\end{align} 
Similar to the arguments for \eqref{heightDG},  after further shrinking $r_j$ according to $P$,  we find  a positive constant $C_2$ independent of $\beta$ such  that
 \begin{align}\label{lastcoefficient}
 T_{ P, \sigma}(r)+T_{D_{\beta}(P), \sigma}(r)\le_{\rm exc} C_2(\mathfrak{X}^+_{\sigma}(r)+\log r) + O( \log    T_{\beta}(r)).
\end{align}  
Hence,
 \begin{align}
N_{P(\beta),  \sigma }(0,r)-N^{(1)}_{P(\beta),  \sigma}(0,r)\le_{\rm exc}  C_1C_2 (\mathfrak{X}^+_{\sigma}(r)+\log r)+  O( \log    T_{\beta}(r) ),
\end{align} 
which contradicts  that fact  that  \eqref{Ru3} fails to hold by taking $c_0\ge C_1C_2 $.
 \end{proof}

 \begin{theorem}\label{proxin}  Let ${\mathcal Y}\subset {\Bbb C}$ be given as above  with   $\sigma$ being associated to the  discs  $\{{\mathbb D}(a_j,  2r_j )\}$.
 Let $G=\sum_{\mathbf i\in I_G}\alpha_{\mathbf i}{\mathbf x}^{\mathbf i}$ be a non-constant   polynomial   in $\mathbb C[x_1,\cdots, x_n]$.  Assume that $G(0,\hdots,0)\ne 0$ and $\deg_{x_i}G=\deg G=d$ for $1\le i\le n$.
  Let $f_0,\dots,f_n$ be  entire functions without common zero which are  units  on $\mathcal Y$. 
 Let $u_i:=f_i/f_0$, $1\le i\le n$, and ${\bf u}=(u_1, \dots, u_n)$.  Assume that  $\sum_{\mathbf i\in J}\alpha_{\mathbf i}{\mathbf u}^{\mathbf i}\ne 0$ for any  non-empty subset  $J$  of $I_G$.
 Then, for every $\epsilon>0$, by shrinking $r_j$ small enough (depending on $\epsilon$, ${\bf u}$ and $G$), we have
 \begin{align}\label{multizeroPara}
 N_{G({\bf u}), \sigma}(0,r)\ge_{\exc}  (\deg G-\epsilon) \cdot \max_{1\le i\le n}\{T_{u_i} (r)\}-c (\mathfrak{X}^+_{\sigma}(r)+ \log r),
 \end{align}
where $c= \frac12 \binom{n+d}{n}  (\binom{n+d}{n}+1)$. 
\end{theorem}  

\begin{proof} 
 Since  $G(0,\hdots,0)\ne 0$ and $\deg_{x_i}G=\deg G=d$ for $1\le i\le n$, we may assume without loss of generality that $G(0,\hdots,0)=1$, and write
 \begin{align}\label{expressG}
G(\mathbf u) =1+ \sum_{1\le i\le n}\alpha_{\mathbf i_i}u_i^d+\sum_{\mathbf i\in I_G\setminus I }\alpha_{\mathbf i}{\mathbf u}^{\mathbf i} 
\end{align}
where $\alpha_{\mathbf i_i}\ne 0$ for $0\le i\le n$ and $I=\{  \mathbf i_0:=(0,0,\hdots,0),\mathbf i_1:=(d,0,\hdots,0),\hdots, \mathbf i_n:=(0,\hdots,0,d)\}$.
 Since  $\sum_{\mathbf i\in J}\alpha_{\mathbf i}{\mathbf u}^{\mathbf i}\ne 0$, for any  non-empty subset  $J$  of $I_G$,
we may use Theorem~\ref{trunborel} to show that  
\begin{equation*}
		 d \max_{1\le i\le n}\{T^M_{u_i,\sigma} (r)\} \le   N_{G(\mathbf u),\sigma}(0,r)+c  (\mathfrak{X}^+_{\sigma}(r)+ \log r) + c^+_{ h_{\mathbf u}, \sigma}(r)+O(\max_{1\leq i\leq n}\log T_{u_i} (r)),
\end{equation*}
where $c= \frac12 \binom{n+d}{n}  (\binom{n+d}{n}+1)$,    and $h_{\mathbf u}$ is a meromorphic function depending on ${\bf u}$.   By Lemma \ref{Ch},  by shrinking $r_j$ small enough if necessary, we can make
$ c^+_{ h_{\mathbf u}, \sigma}(r)\leq O(1)$. In addition, we  can shrink $r_j$ small enough such that (\ref{Amero}) and the Proposition \ref{exhuastion} for $u_i, 1\leq i\leq n$, hold, i.e., $|T^M_{u_i,\sigma} (r)- T_{u_i} (r)|\leq \log r+O(1)$.
This proves the theorem.
\end{proof}

\begin{theorem}[The abc Theorem on parabolic Riemann surfaces]\label{mainthm} Let ${\mathcal Y}\subset {\Bbb C}$ be given with   $\sigma$ being associated to the discs $\{{\mathbb D}(a_j,  2r_j )\}$.  
Let $G$ be a non-constant   polynomial in $\mathbb C[x_1,\hdots,x_n]$ with no monomial  factors and no repeated factors.
Then, for any $\epsilon>0$,   there exist a positive real number $c_0$, a proper Zariski closed subset $W$ of $\mathbb{A}^n(\mathbb C)$ expressible as the zero locus of a finite set
$\Sigma\subset \mathbb C[x_1,\ldots,x_n]$ with the following property: 
   Let $f_0,\dots,f_n$ be entire functions without common zeros, which are units on $\mathcal{Y}$.
For $1 \le i \le n$, set $u_i := f_i/f_0$ and write ${\bf u} = (u_1,\dots,u_n)$.
Assume that $H({\bf u}) \neq 0$ for all $H \in \Sigma$.  Then,  by shrinking $r_j$ small enough (depending on $\epsilon$, ${\bf u}$ and $G$), we have
 \begin{enumerate}
 \item[{\rm(a)}]  $N_{G(\mathbf u),\sigma}(0,r)-N^{(1)}_{G(\mathbf u),\sigma}(0,r)) \le_{\rm exc} \epsilon\max_{1\le i\le n}\{T_{u_i}(r)\} +c_0(\mathfrak{X}^+_{\sigma}(r)+ \log r).$
 \item[{\rm(b)}]  If $G(0,\hdots,0)\ne 0$ and $\deg_{x_i}G=\deg G=d$ for $1\le i\le n$, then 
 \begin{align*} 
  (\deg G -\epsilon)\cdot  \max_{1\le i\le n}\{T_{u_i}(r)\}\leq_{\rm exc}  N^{(1)}_{G(\mathbf u),\sigma}(0,r)) + c_1(\mathfrak{X}^+_{\sigma}(r)+ \log r).
 \end{align*}
  \end{enumerate}
  Here, $c_0$   and  $c_1$  can be  effectively bounded from above in terms of $\epsilon$, $n$, and the degree of $G$.    Moreover,  the finite set
$\Sigma\subset \mathbb C[x_1,\ldots,x_n]$ satisfies: {\rm(Z1)} $\Sigma$ depends on $\epsilon$ and $G$, but not on ${\bf u}$ or $\mathcal Y$, and can be determined explicitly,  {\rm(Z2)} $\vert \Sigma\vert$ and the degree of each polynomial in $\Sigma$ can be effectively bounded from above in terms of $\epsilon$, $n$, and the degree of $G$.  
 \end{theorem}
\begin{proof}
The proof  follows from Theorem \ref{mainthmK} and Theorem \ref{proxin}.
 Let $W = Z \cup Z_1 \cup [G=0]$, where $Z$ and $Z_1$ are the exceptional sets arising from 
Theorem~\ref{mainthmK} and Theorem~\ref{proxin} for $\frac{\epsilon}{2}$, respectively.  
 We express the zero locus of $W$ as the vanishing set of a finite collection 
$\Sigma \subset \mathbb{C}[x_1,\ldots,x_n]$, and assume that 
$H({\bf u}) \ne 0$ for all $H \in \Sigma$.
 By Theorem \ref{mainthmK}, we obtain 
 $$
 N_{G(\mathbf u),\sigma}(0,r)-N^{(1)}_{G(\mathbf u),\sigma}(0,r))\le\frac \epsilon2\max_{1\le i\le n}\{T_{u_i}(r)\} +c_0(\mathfrak{X}^+_{\sigma}(r)+ \log r).
 $$
 On the other hand, from  \eqref{multizeroPara} in Theorem \ref{proxin}, 
$$N_{G({\bf u}), \sigma}(0,r)\ge_{\exc}  (\deg G-\epsilon) \cdot \max_{1\le i\le n}\{ T_{u_i} (r) \}-c (\mathfrak{X}^+_{\sigma}(r)+ \log r).$$ 
Thus we derive
\begin{align}\label{N1} 
N^{(1)}_{G(\mathbf u),\sigma}(0,r))\ge_{\exc}  \left(\deg G- \frac{\epsilon}2\right) \max_{1\le i\le n}\{T_{u_i}(r)\}-c(\mathfrak{X}^+_{\sigma}(r)+ \log r) .
\end{align}
Let $c_1=c_0+c$,  then 
\begin{align*}  
N^{(1)}_{G(\mathbf u),\sigma}(0,r))\ge_{\exc}  (\deg G- \epsilon)\cdot \max_{1\le i\le n}\{T_{u_i} (r)\}+c_1(\mathfrak{X}^+_{\sigma}(r)+ \log r).
\end{align*}
\end{proof}

\section{Proof of the main theorem}\label{proofmain}
\begin{proof}[Proof of  the Main Theorem]
Let $f=(f_0,\hdots,f_n):\mathbb C\mapsto\mathbb P^n(\mathbb C)$, where $f_0,\hdots,f_n$ are entire functions without common zero. We may assume $f_i$ is not identically zero by letting $Z$ contain  all coordinate hyperplanes. When $f$ is rational,  the assertion follows directly from Theorem \ref{rational}, as in this case we have 
 $$
 N_{G(f)}(0,r)=N (G(f))\log r, \quad
 N^{(1)}_{G(f)}(0,r)=N^{(1)} (G(f))\log r, \quad\text{ and }  T_{f}(r)=(\deg f)\log r
 $$
 for $r$ sufficiently large.  Therefore, it remains to consider the situation when $f$ is not rational.  In this case, we have
 $$\log r=o(T_f(r)).$$
 
Let $D:=\sum_{i=0}^n H_i$, where $H_i:=[x_i=0]$, $0\le i\le n$, are the coordinate hyperplanes.  We consider  ${\mathcal Y_{f}}:={\Bbb C}\backslash  f^{-1}(D)$ with 
  $\sigma$ being associated to the discs $\{{\mathbb D}(a_j,  2r_j )\}$, as being constructed in Section  \ref{Jensen}.
  Furthermore, by further shrinking $r_j$, we may assume that 
\begin{align}\label{charCompare2}
|T_{u_i}(r)- T^{M}_{ u_i, \sigma}(r)|\le \log r + O(1)
\end{align}
for $1\le i\le n$ by \eqref{Amero} and Proposition \ref{exhuastion} (ii), 
and 
\begin{align}\label{estexhausion}
 \frak X_{\sigma}^+(r)\le   \sum_{i=0}^n  N_f^{(1)}(H_i,r)+O(\log r)   
\end{align}
by \eqref{sigmaestimae}. 
Since  ${\rm mult}_{z_0}(f^*H_i) \ge  \ell$ for all $i$ with $0\leq i\leq n$ and
for every $z_0 \in {\Bbb C}$, we have 
 $N_f^{(1)}(H_i,r) \leq {1\over \ell} T_f(r)$ for $i=0, 1, \dots, n$.  Thus 
 \begin{align}\label{estexhausion2}
 \frak X_{\sigma}^+(r)\le \frac{n+1}{\ell} T_f(r)+ O(\log r),
\end{align}
where $\ell$ is a positive integer to be determined.

   Let $u_i:=\frac{ f_i}{ f_0}$. 
Then $u_i, 1\leq i\leq n$,  are units on  $\mathcal Y_f $.  
  Let  $\tilde G\in\mathbb C[x_1,\hdots,x_n]$ be defined by  $\tilde G(x_1,\hdots,x_n)=G(1,x_1,\hdots,x_n)$.
Then $f_0^{\deg G}\cdot  {\tilde G}({\mathbf u})= G(f).$
Hence it is obvious that
 \begin{align}\label{compareCounting}
  N^{(1)}_{\tilde G( {\mathbf u}),\sigma}(0,r)=N^{(1)}_{G(f),\sigma}(0,r).
 \end{align}   
Let $\epsilon>0$, and let  $W$ be the proper Zariski closed subset $W$ of $\mathbb A^n$ determined in  Theorem \ref{mainthm}.
 We express the zero locus of $W$ as the vanishing set of a finite collection 
$\Sigma \subset \mathbb{C}[x_1,\ldots,x_n]$, and assume that 
$H({\bf u}) \ne 0$ for all $H \in \Sigma$.
 Then, by Theorem \ref{mainthm}, together with \eqref{charCompare2} and \eqref{compareCounting}, we obtain, after shrinking  $r_j$ if necessary, that there exists a positive real $c_0$ (depending only on $\epsilon$, $n$, and $\deg G$) such that 
 \begin{align}\label{simplezero}   N_{G(f),\sigma}(0,r)-N^{(1)}_{G(f),\sigma}(0,r) \le_{\exc}\frac{\epsilon}3\max_{1\le i\le n}\{T_{u_i}(r)\} +c_0(\mathfrak{X}_{\sigma}^+(r)+\log r),
\end{align}
and 
\begin{align}\label{truncatezero} 
 \left(\deg G -\frac{\epsilon}2\right)\cdot  \max_{1\le i\le n}\{T_{u_i}(r)\}\le_{\exc}  N^{(1)}_{G(f),\sigma}(0,r)+c_0(\mathfrak{X}_{\sigma}^+(r)+\log r).
 \end{align}
 Using  \eqref{estexhausion2}  and the fact that $\log r=o(T_f(r))$, by taking $\ell >3c_0(n+2)\epsilon^{-1}$,
 we have that
\begin{align}\label{c0X} 
c_0(\mathfrak{X}_{\sigma}^+(r) +\log r)  \le {\epsilon/3}\max_{1\le i\le n}\{T_{u_i}(r)\}.
\end{align}
 Since $B_t^{\sigma}\subset \mathbb{D}_t$ for any $t\ge 1$, $N^{(1)}_{G(f),\sigma}(0,r)\le  N^{(1)}_{G( f)}(0,r)$.
 It follows immediately  from  \eqref{truncatezero} and \eqref{c0X}  that  
  \begin{align}\label{truncatef} 
N^{(1)}_{G( f)}(0,r)\ge_{\rm exc}   \left(\deg G -\frac{5\epsilon}6\right)\cdot  \max_{1\le i\le n}\{T_{u_i}(r)\}\ge   (\deg G - \epsilon )T_{ f }(r).
 \end{align}
 This proves (ii).

  It remains to prove (i), which follows from (ii), i.e., from \eqref{truncatef}  together with the First Main Theorem.
\begin{align*}
N_{G( f)}(0,r)-N^{(1)}_{G( f)}(0,r)
&\le_{\exc}  \deg G\cdot  T_f(r)-(\deg G - \epsilon )T_{ f }(r)\le_{\exc} \epsilon T_f(r).
\end{align*}
 
  Finally, we can conclude the proof by taking $Z$ to be the Zariski closure of $W$ in $\mathbb P^n({\Bbb C})$.
 \end{proof}
\section{Toric Varieties}\label{toricSection}
\begin{proof}[Proof of Theorem \ref{toric}]
We recall the following  setup  of finding a natural finite open covering of $X$ from the proof of \cite[Theorem 4.4]{levin_gcd}.
Let $\Sigma$ be the fan corresponding to the smooth projective toric variety $X$. Then there is a finite affine  covering $\{X_{\eta} \}$ of $X$, where $\eta\in \Sigma$ is a n-dimensional smooth cone with an isomorphism $i_{\eta}: X_{\eta}\to\mathbb A^n$. 
 This isomorphism restricts to an automorphism of $\mathbb G_m^n$, where we identify $\mathbb G_m^n\subset X_{\eta}$ naturally as a subset of $X$ and 
$\mathbb G_m^n\subset\mathbb A^n$ in the standard way such that $\mathbb A^n\setminus \mathbb G_m^n$ consists of the affine coordinate hyperplanes $\{x_i=0\}$, $1\le i\le n$.  
Under this identification, an  entire orbifold curve $f$ into $(X,\Delta)$ corresponds to a collection of $n$-tuples of meromorphic functions ${i}_{\eta}(f)=(u_{\eta,1},\hdots,u_{\eta,n})$ such that the zero multiplicity of each $u_{\eta,i}$, $1\le i\le n$, is at least $\ell$ if not zero.  If  $z_0\in\mathbb C$ is a pole of some  $u_{\eta,i}$, $1\le i\le n$, then $f(z_0)$ is in a component of $D_0$, the boundary of $X$, with multiplicity
$$
-\min\{0,v_{z_0}(u_{\eta,1}),\hdots, v_{z_0}(u_{\eta,n})\} \ge \ell.
$$ 
Let $ \widetilde{ {i}_{\eta}}: X_{\eta}\to\mathbb P^n({\Bbb C})$ be the composition  of $i_{\eta}$ with the natural embedding of $\mathbb A^n$ to $\mathbb P^n({\Bbb C})$. 
Then $ \widetilde{ {i}_{\eta}}(f)=(f_{\eta,0},\hdots,f_{\eta,n})  =[1:u_{\eta,1}:\cdots:u_{\eta,n}]:\mathbb C\to \mathbb P^n({\Bbb C})$,  
where $ f_{\eta,0},\hdots,f_{\eta,n}$ are entire functions with no common zero.
Furthermore, the pullback $(\widetilde{ {i}_{\eta}}^{-1})^*(D|_{X_{\eta}})$ of $D$ to $\mathbb P^n({\Bbb C})$ is defined by some nonconstant homogeneous polynomial $F_{\eta}\in \mathbb C[x_0,\hdots,x_n]$ with no monomial and repeated factors and   $\deg_{X_i}F_{\eta}=\deg F_{\eta}$ for each $0\le i\le n$ since $D$ is reduced and is in general position with the boundary of $\mathbb G_m^n$ in $X$.   
 By  \cite[Proposition 12.11]{Vojta}, there exist a non-zero constant  $c_{\eta,A}$ and  a proper closed subset $Z_{\eta,A}\subset X_{\eta}$, depending on $\eta$ and  $A$ such that 
\begin{align}\label{htsigma}
 T_{ \widetilde{ {i}_{\eta}}(f) } (r)\le c_{\eta,A} T_{A,f}(r)+O(1) 
\end{align}
if the image of ${i}_{\eta}(f)$ is not contained in  $Z_{\eta,A}$.
Let $M$ be the number of $X_{\eta}$ in $\Sigma$.  We may now apply the Main Theorem  to each $X_{\eta}$.  Let $\epsilon>0$.  Then there exists a proper Zariski closed subset  $W_{\eta}\subset X_{\eta}$ 
for each $\eta\in \Sigma$  such that for sufficiently large integer $\ell$,  we have
\begin{align}\label{multieta}
N_{F_{\eta}(\widetilde{{i}_{\eta}}(f))}(0,r)- N^{(1)}_{F_{\eta}(\widetilde{{i}_{\eta}}(f))}(0,r)  
\le_{\rm exc}  \frac{\epsilon}{Mc_{\eta,A}} T_{ \widetilde{ {i}_{\eta}}(f) } (r),
\end{align}
and 
\begin{align}\label{truncateeta}
N^{(1)}_{F_{\eta}(\widetilde{{i}_{\eta}}(f))}(0,r) 
\ge_{\rm exc}   (\deg F_{\eta}-\frac{\epsilon}{Mc_{\eta,A}}) T_{ \widetilde{ {i}_{\eta}}(f) } (r),
\end{align}
for all   $f$ such that  the image of $f$ restricting to $X_{\eta}$ is not contained in $W_{\eta}$.  
 Since $\deg_{X_i}F_{\eta}=\deg F_{\eta}$ for each $0\le i\le n$, it implies that 
$\deg F_{\eta}\cdot T_{ \widetilde{ {i}_{\eta}}(f) } (r)=T_{F_{\eta}(\widetilde{{i}_{\eta}}(f))}(r)+O(1).$
Then by the first main theorem and \eqref{truncateeta}, we have 
\begin{align}\label{proxi1}
m_{F_{\eta}(\widetilde{{i}_{\eta}}(f))}(0,r) 
\le_{\rm exc} \frac{\epsilon}{Mc_{\eta,A}}  T_{ \widetilde{ {i}_{\eta}}(f) } (r) .
\end{align}

Let $W$ be the union of  $W_{\eta}$ and $Z_{\eta,A}$ for $\eta\in\Sigma$.  Then
 \begin{align}\label{countingD}
N_f(D,r)-N_f^{(1)}(D,r)\le \sum_{\eta\in \Sigma}N_{F_{\eta}(\widetilde{{i}_{\eta}}(f))}(0,r)-N^{(1)}_{F_{\eta}(\widetilde{{i}_{\eta}}(f))}(0,r) \le_{\rm exc}   \epsilon T_{A,f}(r),
\end{align}
and
\begin{align}\label{proxi2}
 m_f (D,r)\le  \sum_{\eta\in \Sigma} m_{F_{\eta}(\widetilde{{i}_{\eta}}(f))}(0,r) \le_{\rm exc} \epsilon T_{A,f}(r)
\end{align}
for all orbifold curve $f:\mathbb C\to  (X,\Delta)$ with multiplicity at least $\ell $ along $D_0$ and $f(\mathbb C)\not\subset W$.
Finally, the assertion (ii) follows  from \eqref{countingD} and the following inequality:
\begin{align*} 
N_f(D,r)=T_{D,f}(r)-m_f (D,r)  \ge_{\rm exc}T_{D,f}(r)-\epsilon T_{A,f}(r)
\end{align*}
by \eqref{proxi2}.
\end{proof}

\begin{proof}[Proof of Theorem  \ref{finitemorphismtoric}]
	 Let $ \pi:   Y\to X$ be a finite morphism.  
We recall from \cite[Lemma 1]{CZ2013} and the proof of \cite[Theorem 6]{GGW} that $R$ is a big divisor since  $(Y,\Delta)$ is   of  general type.
Let $R_0$ be an irreducible component of $R$. Without loss of generality, we let $R_0=R$.  Otherwise, we simply repeat the following steps for each irreducible component.  Let $D=\pi(R)$, which is in general position with $D_0$ by assumption.
Now, consider $f : \mathbb C \to  (Y,\Delta) $  as a nonconstant orbifold entire curve.	 
Since $  \pi^*D$ has multiplicity at least 2 along $R$, we have
\begin{align}\label{coutingzero1}
		N_{f}( R,r )\le N_{\pi\circ f}( D,r )-N^{(1)}_{\pi\circ f}(D,r ).
	\end{align}	
 
We are now in a position to apply Theorem \ref{toric} for a small $\epsilon >0$ (to be determined later) and an ample divisor $A$ on $X$.  Consequently, there exists a positive integer $\ell_1$ and a proper Zariski closed subset $W_1$ of $X$ such that if the orbifold $ (X,\Delta_0)$  has  multiplicities at least $\ell_1$ along $D_0$,	then  
\begin{align}\label{gcdF}
	N_{\pi\circ f}( D,r )-N^{(1)}_{\pi\circ f}(D,r ) \le_{\exc}  \epsilon T_{A,\pi\circ f}(r),
\end{align}
and 	
\begin{align}\label{countingF}
	 T_{D,\pi\circ f}(r) \le_{\rm exc} N_{\pi\circ f}^{(1)}(D,r) +\epsilon T_{A,\pi\circ f}(r) 
\end{align}
	if   the image of $\pi\circ f$ is not contained in $W_1$.
Assuming that $\ell\ge\ell_1$	 and the image of $f$ is not contained in $\pi^{-1}(W_1)$,
	then by   \eqref{coutingzero1} and \eqref{gcdF}, we have 
	\begin{align}\label{zeroZ}
		N_{f}( R,r )\le_{\exc} \epsilon T_{A,\pi\circ f}(r). 
	\end{align}
 Since $  R \le   \pi^* (D)$ (as divisors),  the functorial property implies that    
\begin{align}\label{Fproxi2}
	 m_{ f}  (R,r) \le	m_{\pi\circ f}(D,r)   +O(1)=T_{D,\pi\circ f}(r)-N_{\pi\circ f}(D,r)  +O(1)
	 	\end{align}	
Together with 	\eqref{countingF}, we have:
\begin{align}\label{FproxiZ}
		m_{f}(R,r) \le\epsilon T_{A,\pi\circ f}(r) +O(1).
	\end{align}	
Combining the two equations     \eqref{zeroZ} and \eqref{FproxiZ} and using the functorial property of characteristic functions, we have: 
\begin{align}\label{charZ}
		T_{R,f}(r) \le 2\epsilon  T_{\pi^*A,  f}(r) +O(1).
	\end{align}	 	
	 	Since $R$ is a big divisor on $Y$, by \cite[Proposition 12.11]{Vojta}, there exists a constant $c $  and a   proper Zariski-closed subset   $W_0$ of $Y$, depending only on $A$ and $R$, such that 
such that 
	\begin{equation}\label{bdd_ample}
		 T_{\pi^*A,  f}(r) \le  c T_{R,f}(r)+O(1),
		 \end{equation}	
	if the image of $\mathbf{f}$ is not contained in $W_0$.
	Combining   \eqref{bdd_ample} with  \eqref{charZ},  we obtain
	$$
	T_{R,f}(r) \le_{\exc} 2 c \epsilon T_{R,f}(r) +O(1).
	$$
	This is not possible  as  $c$ is independent of   $\epsilon $, and $\epsilon $ can be taken sufficiently small.  In conclusion, we have found a sufficiently large integer $\ell$ such that if $(Y,\Delta)$  has  multiplicities at least $\ell$ along $D_0$, then the image of any orbifold entire curve $f$ on $(Y,\Delta)$ is contained in 
	$\pi^{-1}(W_1)\cup W_0$. 
	\end{proof}
\noindent\textbf{Acknowledgements.} 
We sincerely thank the referees for their careful reading of the manuscript and for their many valuable comments and suggestions. We also thank the first author's student, Zhe Wang, for helpful discussions and valuable suggestions.

\end{document}